\documentclass[10pt,reqno]{amsart}
\usepackage{amsmath,graphicx,color}
\usepackage[utf8]{inputenc}
\usepackage[T1]{fontenc}
\usepackage{ifthen}
\usepackage{cite}
\usepackage{dsfont}
\usepackage{amssymb,latexsym}
\usepackage{subfigure}
\usepackage{hyperref}
\hypersetup{
urlcolor=black, 
  menucolor=black, 
  citecolor=black, 
  anchorcolor=black, 
  filecolor=black, 
  linkcolor=black, 
  colorlinks=true,
}
\usepackage{comment}
\usepackage[numbers,sort&compress]{natbib} 

\numberwithin{equation}{section}

\newcommand{\wt }{\widetilde}

\newcommand{\calK}{{\mathcal K}}
\newcommand{\bfk}{{\mathbf k}}

\newcommand{\bfA}{{\mathbf A}}

\newcommand{\bfR}{{\mathbf R}}

\newcommand{\bfGamma}{{\mathbf \Gamma}}
\newcommand{\bfSigma}{{\mathbf \Sigma}}

\newcommand{\bfI}{{\mathbf I}}

\newcommand{\bfx}{{\mathbf x}}

\newcommand{\beao}{\begin{eqnarray*}}
\newcommand{\eeao}{\end{eqnarray*}}
\newcommand{\beam}{\begin{eqnarray}}
\newcommand{\eeam}{\end{eqnarray}}
\newcommand{\barr}{\begin{array}}
\newcommand{\earr}{\end{array}}

\definecolor{darkblue}{rgb}{.1, 0.1,.8}
\definecolor{darkgreen}{rgb}{0,0.8,0.2}
\definecolor{darkred}{rgb}{.8, .1,.1}
\newcommand{\bco}{\begin{corrolary}}
\newcommand{\eco}{\end{corrolary}}

\newcommand{\E}{\mathbb{E}}
\renewcommand{\P}{\mathbb{P}}
\newcommand{\1}{\mathds{1}}
\newcommand{\R}{\mathbb{R}}
\newcommand{\N}{\mathbb{N}}

\newcommand{\bfS}{{\mathbf S}}
\newcommand{\bfB}{{\mathbf B}}

\newcommand{\Var}{\operatorname{Var}}

\DeclareMathOperator{\e}{e}

\newcommand{\bfz}{{\mathbf z}}
\newcommand{\X}{{\mathbf X}}
\newcommand{\Y}{{\mathbf Y}}

\newcommand{\dint}{\,\mathrm{d}}

\newcommand{\vep}{\varepsilon}
\newcommand{\nto}{n \to \infty}

\newcommand{\lhs}{left-hand side}
\newcommand{\rhs}{right-hand side}

\newcommand{\tr}{\operatorname{tr}}

\newcommand{\diag}{\operatorname{diag}}

\newcommand{\MP}{Mar\v cenko--Pastur }

\def\tag{\refstepcounter{equation}\leqno }

\newtheorem{lemma}{Lemma}[section]

\newtheorem{theorem}[lemma]{Theorem}
\newtheorem{proposition}[lemma]{Proposition}

\newtheorem{corollary}[lemma]{Corollary}

\newtheorem{remark}[lemma]{Remark}

\newcommand{\cid}{\stackrel{d}{\rightarrow}}
\newcommand{\cip}{\stackrel{\P}{\rightarrow}}

\newcommand{\eid}{\stackrel{d}{=}}

\begin{document}
\bibliographystyle{acm}
\title[Log determinant of heavy-tailed correlation matrix]
{Log determinant of large correlation matrices under infinite fourth moment}
\thanks{Johannes Heiny's research is supported by the Deutsche Forschungsgemeinschaft (DFG) via RTG 2131 High-dimensional Phenomena in Probability -- Fluctuations and Discontinuity.}
\author[J. Heiny]{Johannes Heiny}
\address{Fakult\"at f\"ur Mathematik,
Ruhruniversit\"at Bochum,
Universit\"atsstrasse 150,
D-44801 Bochum,
Germany}
\email{johannes.heiny@rub.de}
\author[N. Parolya]{Nestor Parolya}
\address{Delft Institute of Applied Mathematics,
Delft Universtity of Technology,
Mekelweg 4,
CD 2628 Delft,
The Netherlands}
\email{n.parolya@tudelft.nl}
\begin{abstract}
In this paper, we show the central limit theorem for the logarithmic determinant of the sample correlation matrix $\bfR$ constructed from the $(p\times n)$-dimensional data matrix $\X$ containing independent and identically distributed random entries with mean zero, variance one and infinite fourth moments. Precisely, we show that for $p/n\to \gamma\in (0,1)$ as $n,p\to \infty$ the \emph{logarithmic law}
\begin{equation*}
\frac{\log \det \bfR -(p-n+\frac{1}{2})\log(1-p/n)+p -p/n}{\sqrt{-2\log(1-p/n)- 2 p/n}} \cid N(0,1)\,
\end{equation*}
 is still valid if the entries of the data matrix $\X$ follow a symmetric distribution with a regularly varying tail of index $\alpha\in (3,4)$. The latter assumptions seem to be crucial, which is justified by the simulations: if the entries of $\X$ have the infinite absolute third moment and/or their distribution is not symmetric,  the logarithmic law is not valid anymore. The derived results highlight that the logarithmic determinant of the sample correlation matrix is a very stable and flexible statistic for heavy-tailed big data and open a novel way of analysis of high-dimensional random matrices with self-normalized entries.
\end{abstract}
\keywords{sample correlation matrix, logarithmic determinant, random matrix theory, heavy tails, infinite fourth moment}
\subjclass{Primary 60B20; Secondary 60F05 60G10 60G57 60G70}

\maketitle

\section{Introduction}\setcounter{equation}{0}

The analysis of the logarithmic determinant has always been of considerable interest in the large dimensional random matrix theory. The investigations of the moments of random determinants trace back to the 1950s (see, Dembo \cite{dembo1989} and references therein). The central limit theorems (CLTs) for the logarithmic determinant of random Gaussian matrices, Wigner matrices and matrices with real independent and identically distributed (i.i.d.) entries with sub-exponential tails were proved by Goodman \cite{goodman1963}, Tao and Vu \cite{tao2012} and Nguyen and
Vu \cite{nguyen2014}, respectively. Girko \cite{girko1979} was the first to state that the result of Goodman  \cite{goodman1963} holds for general random matrices under the additional assumption that the fourth moment of the entries is equal to three (normal-like moments of order four). This CLT was named as Girko's logarithmic law or simply \emph{logarithmic law}.  Moreover, twenty years later Girko \cite{girko1998refinement} using an elegant method of perpendiculars partially proved that the CLT for the logarithmic determinant holds in a very generic case under the existence of the $4+\varepsilon$ moments for some small $\varepsilon>0$.  Nguyen and Vu \cite{nguyen2014} show a refined and more transparent proof of this claim assuming a much stronger condition of sub-exponential tails for the random matrix entries and providing additionally the rate of convergence of the logarithmic determinant of the sample covariance matrix. In case the stochastic representation of the logarithmic determinant is available, the large/moderate deviation results are proved in \cite{grote:kabluchko:thaele:2019}, whereas fast Berry--Esseen bounds were recently provided by \cite{heiny:johnston:prochno:2022}.

Consider a random sample $\bfx_1\ldots,\bfx_n$ from a $p$-dimensional distribution collected into a $p\times n$ random data matrix $\X$. For statistical applications the logarithmic determinants of the sample covariance matrix $\bfS=n^{-1} \X\X^{\top}$ and the sample correlation matrix $\bfR=\{\diag(\bfS)\}^{-1/2}\, \bfS\{\diag(\bfS)\}^{-1/2}$ are of vital importance. They allow efficient inferential procedures on the structure of the true covariance/correlation matrices (see, the monographs of Anderson \cite{anderson2003} and Yao, Zheng and Bai \cite{yaobaizheng2015}).
 In particular, the determinant of the sample correlation matrix has numerous applications in stochastic geometry as it is proportional to the volume of the hyperellipsoid constructed from standardized vectors, see \cite{nielsen1999}. Furthermore, the determinant of $\bfR$ is the well-known likelihood ratio statistic for testing the independence of the elements of the random vector in case of multivariate normality of the columns of the data matrix, see, e.g., \cite{jiang2013, bodnar2019testing} and references therein. 

A wide variety of results have been obtained for the large dimensional sample covariance matrix $\bfS$, e.g., \MP~law/equation in \cite{marchenko:pastur:1967, silversteinchoi1995}, CLT for linear spectral statistics in \cite{baisil2004} and Tracy-Widom law in \cite{elkaroui2007}, to mention a few. For the sample correlation matrix $\bfR$, the situation gets more complicated because of the specific nonlinear dependence structure caused by the normalization $\{\diag(\bfS)\}^{-1/2}$,  which makes the analysis of this random matrix quite challenging. In case the elements of the data matrix $\X$ are i.i.d. with zero mean, variance equal to one and finite fourth moment it is shown by Jiang \cite{jiang2004b} (see, also \cite{bai:zhou:2008},\cite{elkaroui:2009} and \cite{heiny2018}) that the \MP~law is still valid for the sample correlation matrix $\bfR$. The asymptotic distribution of the largest eigenvalue of $\bfR$ is proved by \cite{bao2012} to obey the Tracy-Widom law. Moreover, the largest and smallest eigenvalues of $\bfR$ converge to the edges of the \MP~density almost surely \cite{heiny2018}. Thus, the ``first order'' properties (almost sure convergence) of the eigenvalues of the sample covariance matrix $\bfS$ and sample correlation matrix $\bfR$ coincide in case the entries of the data matrix $\X$ possess at least finite second moments (see \cite{heiny:yao:2021}). This observation changes if ``second order'' properties (such as CLTs) are of interest. To illustrate this fact, we compare the CLTs for the logarithmic determinants of $\bfS$ and $\bfR$ under finite fourth moment assumption.

  The logarithmic law of the large sample covariance matrices can be deduced from the work of Bai and Silverstein \cite{baisil2004} for the linear spectral statistics $\tr(f(\bfS))$ with a test function $f(x)=\log(x)$ in case $p$ the number of columns of the data matrix is smaller than $n$ the number of its rows and both tend to infinity such that their ratio tends to a constant, i.e., $p/n\to\gamma\in (0,1)$, as $\nto$. More precisely, Wang and Yao \cite{wang2013} show that if the i.i.d. entries of the data matrix $\X=(X_{ij})_{1\le    i\le p; 1\le j  \le n}$ satisfy $\E(X_{11})=0$, $\Var(X_{11})=1$ and $\E(X^4_{11})<\infty$, the following logarithmic law for its corresponding sample covariance matrix $\bfS$ is valid
\begin{eqnarray}\label{loglaw:cov}
  \frac{\log\det\bfS - (p-n+1/2)\log(1-p/n) +p-\frac{1}{2}\left[\E(X^4_{11})-3\right]p/n}{\sqrt{-2\log(1-p/n)+\left[\E(X^4_{11})-3\right]p/n}} \cid N(0,1)\,, \quad \text{as $n\to\infty$}\,.
\end{eqnarray}
Later on,  Bao, Pan and Wang \cite{bao:pan:zhou:2015} and Wang, Han and Pan \cite{WangHanPan2018} proved a similar CLT for the logarithmic determinant of the sample covariance matrices in case $p/n\to1$ and $p\leq n$ under finite fourth moments.

For the sample correlation matrix $\bfR$ the situation is more involved. The first generic result for the linear spectral statistics of $\tr(f(\bfR))$ for some test function $f(\cdot)$ was proved in \cite{gao2017} under existence of the fourth moment and it states that taking $f(x)=\log(x)$ for $p/n\to\gamma<1$ one gets
\begin{equation}\label{loglaw:corr}
\frac{\log \det \bfR -(p-n+\frac{1}{2})\log(1-p/n)+p -p/n}{\sqrt{-2\log(1-p/n)- 2 p/n}} \cid N(0,1)\,, \quad\text{as $n\to\infty$}\,.
\end{equation}
Surprisingly, the latter logarithmic law is quite different from \eqref{loglaw:cov}, especially the dependence on the fourth moment is not present in \eqref{loglaw:corr}, which indicates that the fourth moment assumption can be eventually weakened (see also \cite{parolya:heiny:kurowicka:2021} and \cite{yang2021}).

In this paper, we contribute to the existing literature by showing that the logarithmic law \eqref{loglaw:corr} is valid for the sample correlation matrix even if the fourth moment of the entries of the data matrix $\X$ is infinite. To the best of our knowledge, this is the first result of this kind. We assume that the i.i.d. elements $X_{ij}$ of $\X$ possess regularly varying tails with index $\alpha\in(3, 4)$ and $X_{ij}\eid -X_{ij}$ (symmetry). In particular, this implies that $\E X_{11}^4 =\infty$ and $\E|X_{11}|^3 <\infty$. Our proof relies on Girko's method of perpendiculars and a CLT for martingale differences together with the exact computation and asymptotics of the moments of the products of self-normalized variables. 

The paper has the following structure: Section~\ref{sec:main} contains notations, assumptions and the main result. In Section~\ref{sec:3}, more precisely in Theorem~\ref{thm:fourthmoment}, we derive an exact formula for the fourth moment of a weighted sum of the components a random vector on the unit sphere. The latter result is of independent interest and can be considered as a first step to generalization of the key lemma for quadratic forms for correlated random vectors of unit length in case of an infinite fourth moment (c.f.~\cite[Lemma 5]{gao2017} and \cite[Lemma 1]{morales:johnstone:mckay:2021}). Asymptotic formulas for the moments of self-normalized variables and the proof of the main theorem are presented in Section~\ref{sec:proofmain}, while the appendix contains some additional auxiliary results.

\section{Main result}\label{sec:main}\setcounter{equation}{0}

Consider a $p$-dimensional population $\bfx=(X_1,\ldots,X_p)\in\R^p$
  where the coordinates  $X_i$  are i.i.d.~ non-degenerated
  random variables with mean zero.
  For a sample $\bfx_1,\ldots,\bfx_n$ from the population we construct the 
  data matrix $\X=\X_n=(\bfx_1,\ldots,\bfx_n)=(X_{ij})_{1\le    i\le p; 1\le j  \le n}$, the sample covariance matrix $\bfS=\bfS_n =n^{-1} \X\X^{\top}$
  and the sample correlation matrix $\bfR$,
  \begin{align}
  \bfR &=\bfR_n =\{\diag(\bfS_n)\}^{-1/2}\, \bfS_n\{\diag(\bfS_n)\}^{-1/2}= \Y \Y^{\top}\,. \label{eq:defRY}
  \end{align}
  Here the standardized  matrix $\Y=\Y_n=(Y_{ij})_{1\le    i\le p; 1\le j  \le n}$ for the sample correlation matrix has
  entries 
  \beam\label{def:R}
  Y_{ij}=Y_{ij}^{(n)}=\frac{X_{ij}}{\sqrt{X_{i1}^2+\cdots+X_{in}^2}}\,,
  \eeam
  which depend on $n$. Throughout the paper, we often suppress the dependence on $n$ in our notation.
	We  consider the
  asymptotic regime
  \begin{equation}\label{Cgamma}
    p=p_n \to \infty \quad \text{ and } \quad \frac{p}{n}\to \gamma\in (0,1)\,,\quad \text{ as } \nto\,. \tag{$C_\gamma$}
  \end{equation}
We assume that $|X_{11}|$ has a regularly varying
tail with index $\alpha>0$, that is 
\beam\label{eq:regvar}
\P(|X_{11}|>x)=  L(x)\, x^{-\alpha}\,,\qquad x>0\,,
\eeam
for a function $L$ that is slowly varying at infinity. Thus, regularly varying distributions possess power-law tails and moments of $|X_{11}|$ of higher order than $\alpha$ are infinite. Typical examples include the Pareto distribution with parameter $\alpha$ and the $t$-distribution with $\alpha$ degrees of freedom.

Now we state the CLT for the logarithmic determinant of the sample correlation matrix $\bfR$ under infinite fourth moment which is the main result of this paper.



\begin{theorem}\label{thm:main}
  Assume \eqref{Cgamma} and that the distribution of
  $X_{11}$ is symmetric and regularly varying with
  index $\alpha \in (3,4)$. Then, as $\nto$, we have 
\begin{equation}\label{eq:main}
\frac{\log \det \bfR -(p-n+\frac{1}{2})\log(1-\frac{p}{n})+p -\frac{p}{n}}{\sqrt{-2\log(1-p/n)- 2 p/n}} \cid N(0,1)\,.
\end{equation}
\end{theorem}
Theorem \ref{thm:main} is proved in Section \ref{sec:proofmain}. 
\begin{figure}[h]
 \centering
    \includegraphics[scale=0.4]{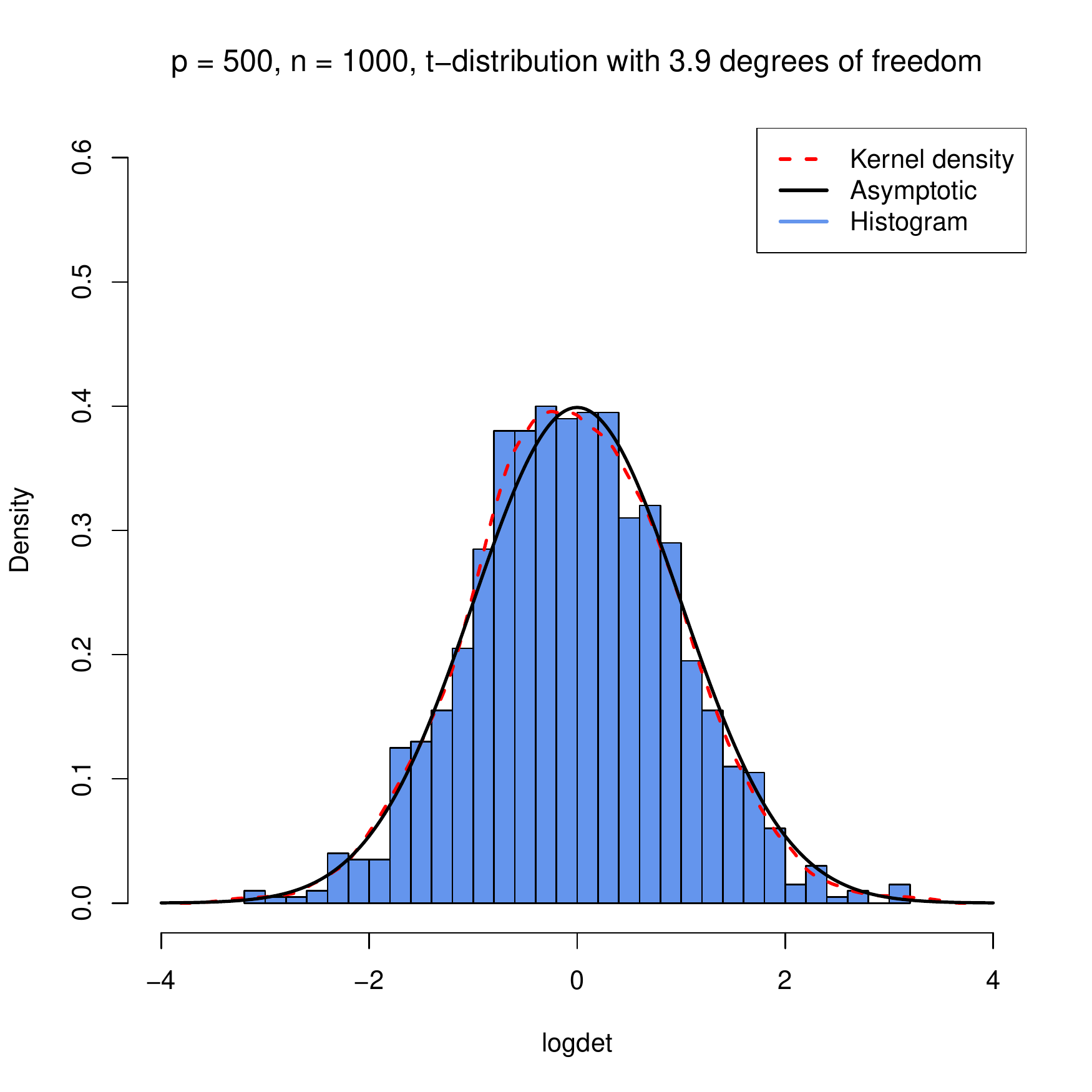} ~~
  \includegraphics[scale=0.4]{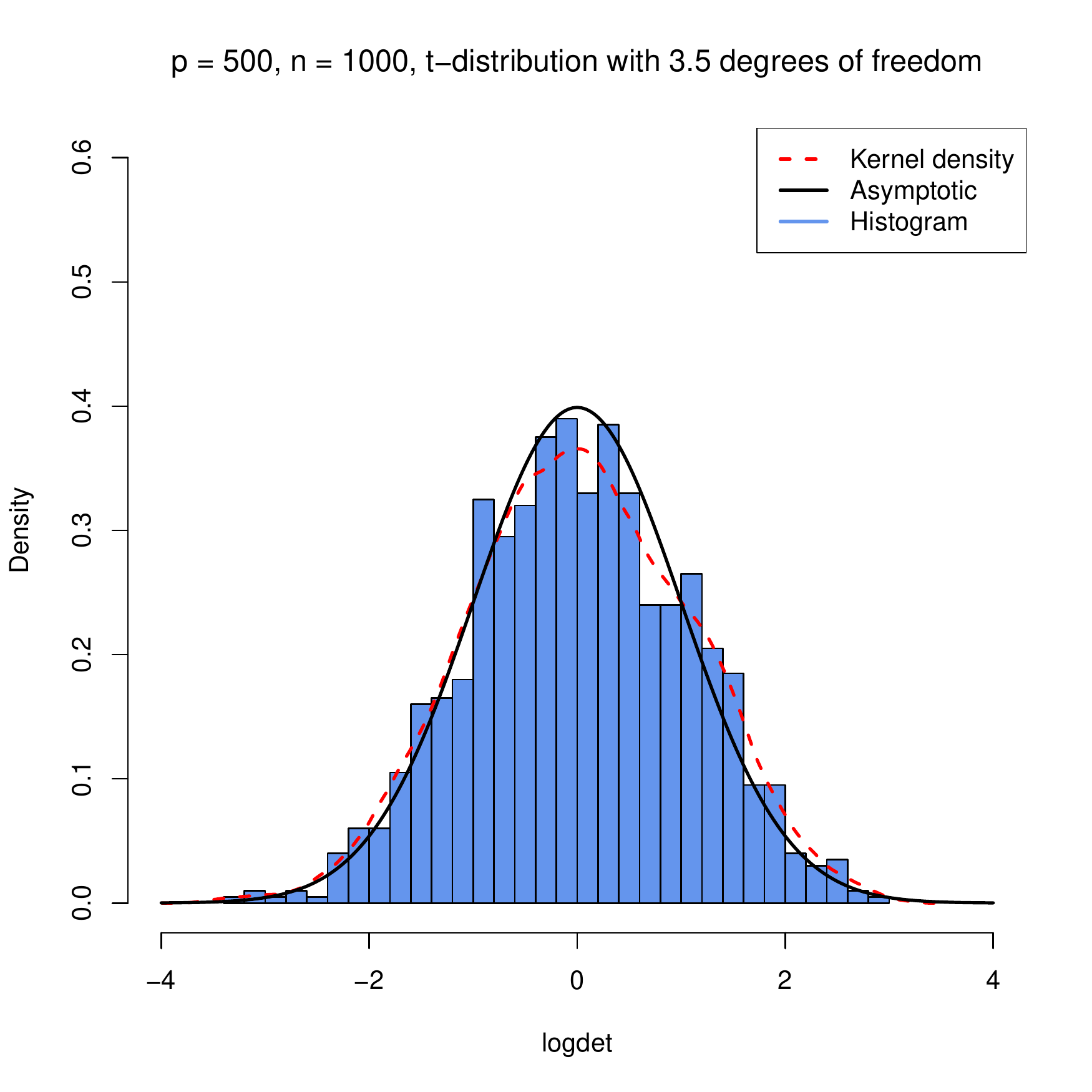}\\
  \includegraphics[scale=0.4]{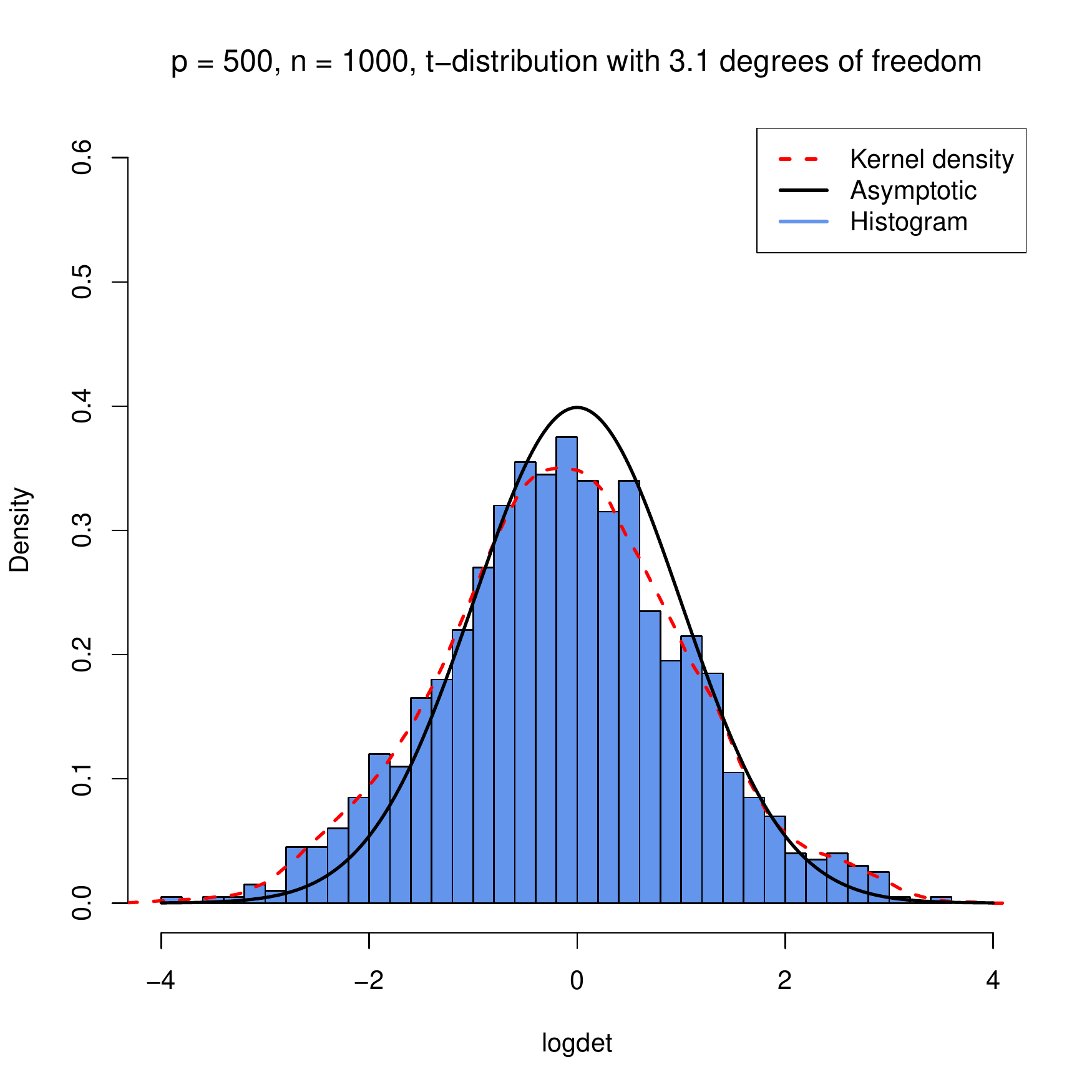}~~
  \includegraphics[scale=0.4]{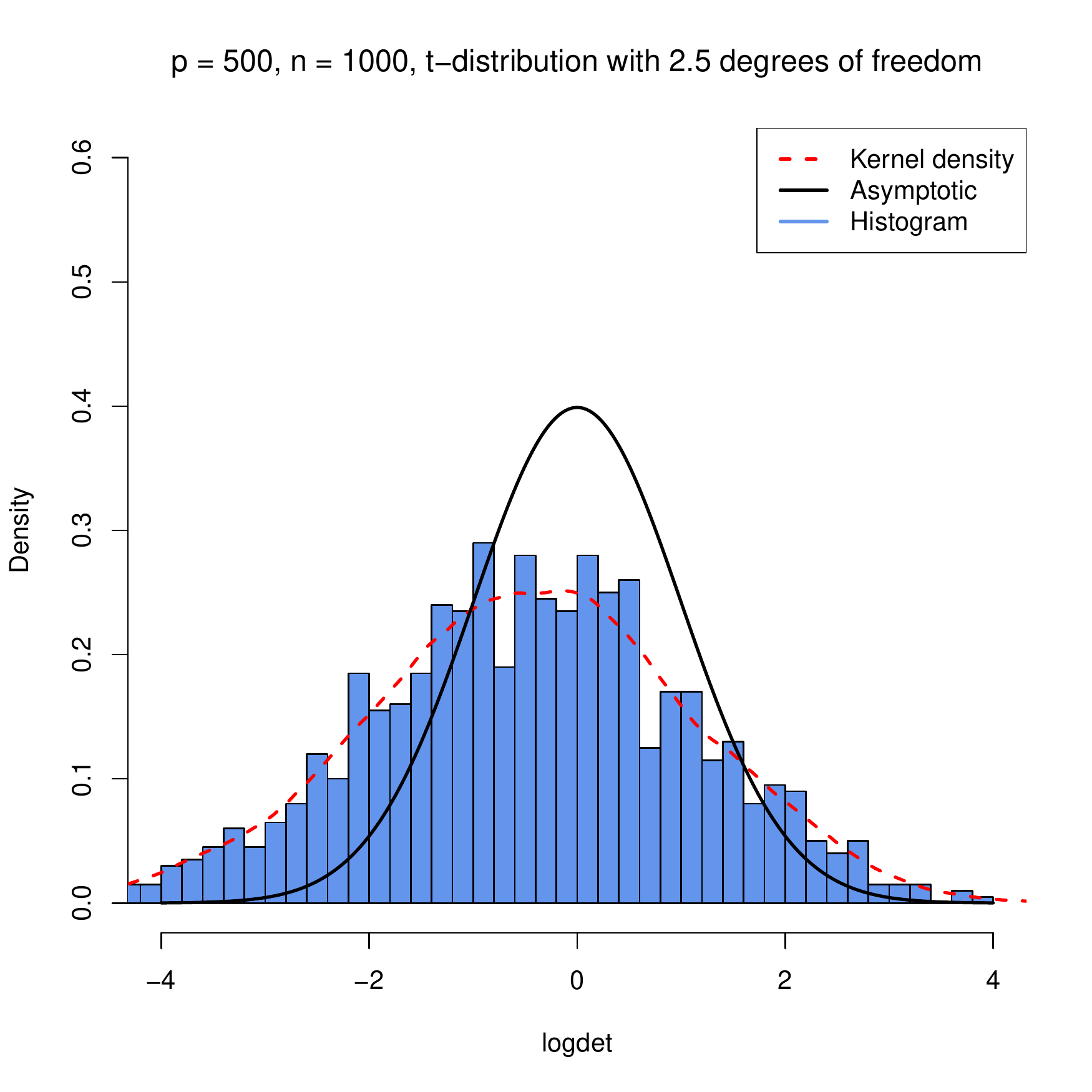}
  \caption{Logarithmic law for $t$ distribution with different degrees of freedom and $p=500$, $n=1000$ with 1000 repetitions.}
  \label{fig1}
\end{figure}
To numerically illustrate the role of the tail index parameter $\alpha$ and the effect of symmetry of $X_{11}$, we provide a small simulation in Figure \ref{fig1} and Figure \ref{fig2}.  First, we simulate the entries of the data matrix $X_{ij}$ independently from a $t$-distribution with different degrees of freedom smaller than four (infinite fourth moment). We observe a perfect fit of  both the histogram and kernel density to the density of the standard normal distribution for all degrees of freedom except $2.5$. Thus, the logarithmic law seems not to be valid in case the third absolute moment of the $t$-distribution is infinite, which is inline with our assumption $\alpha>3$. In the latter case the kernel density still resembles the normal density but has a significantly larger variance, which indicates that the case $\alpha\in(2,3)$ should be investigated separately in the future. The effect of a larger variance becomes more pronounced if we decrease the tail parameter of the observations $X_{ij}$ even further. 

\begin{figure}[h]
 \centering
    \includegraphics[scale=0.4]{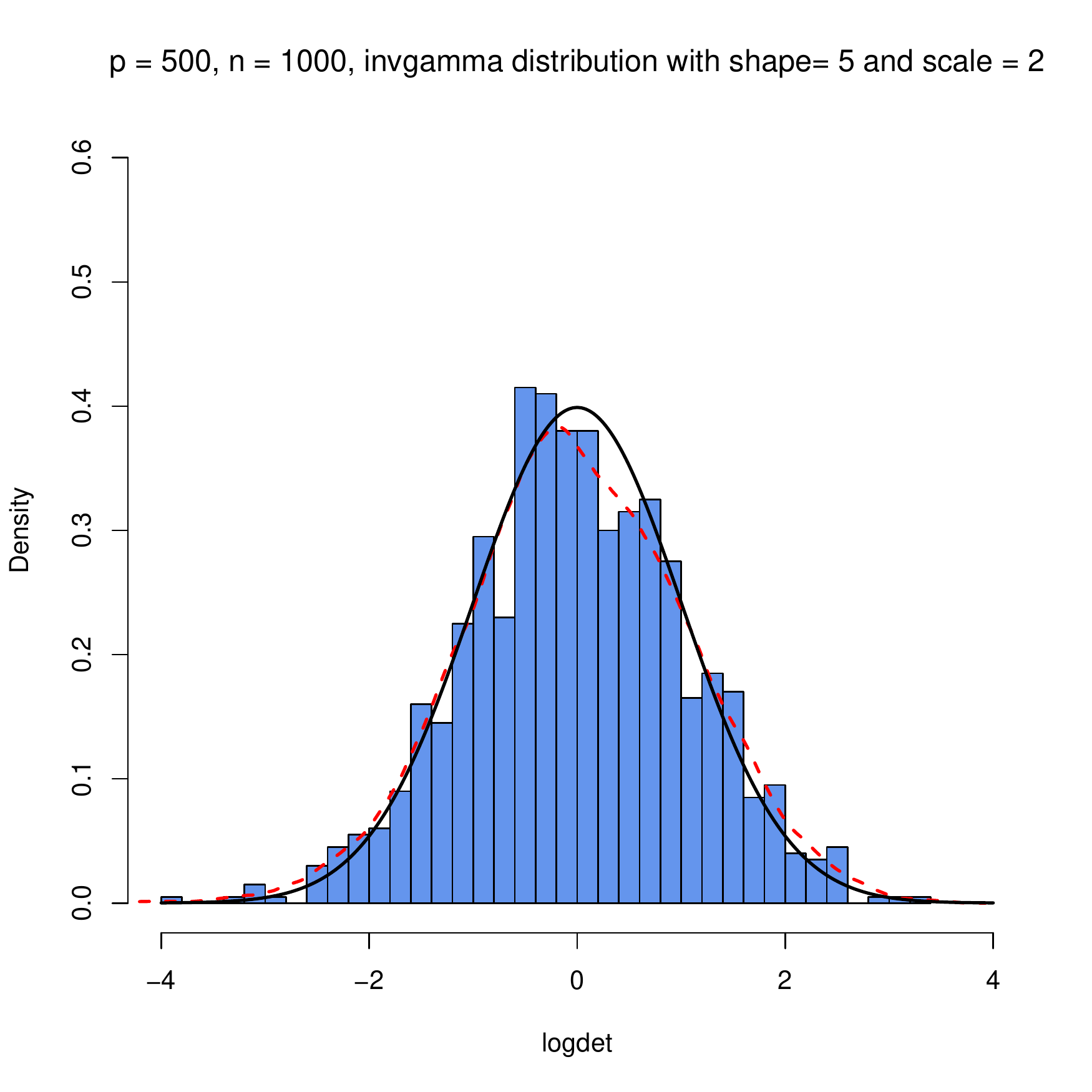} ~~
  \includegraphics[scale=0.4]{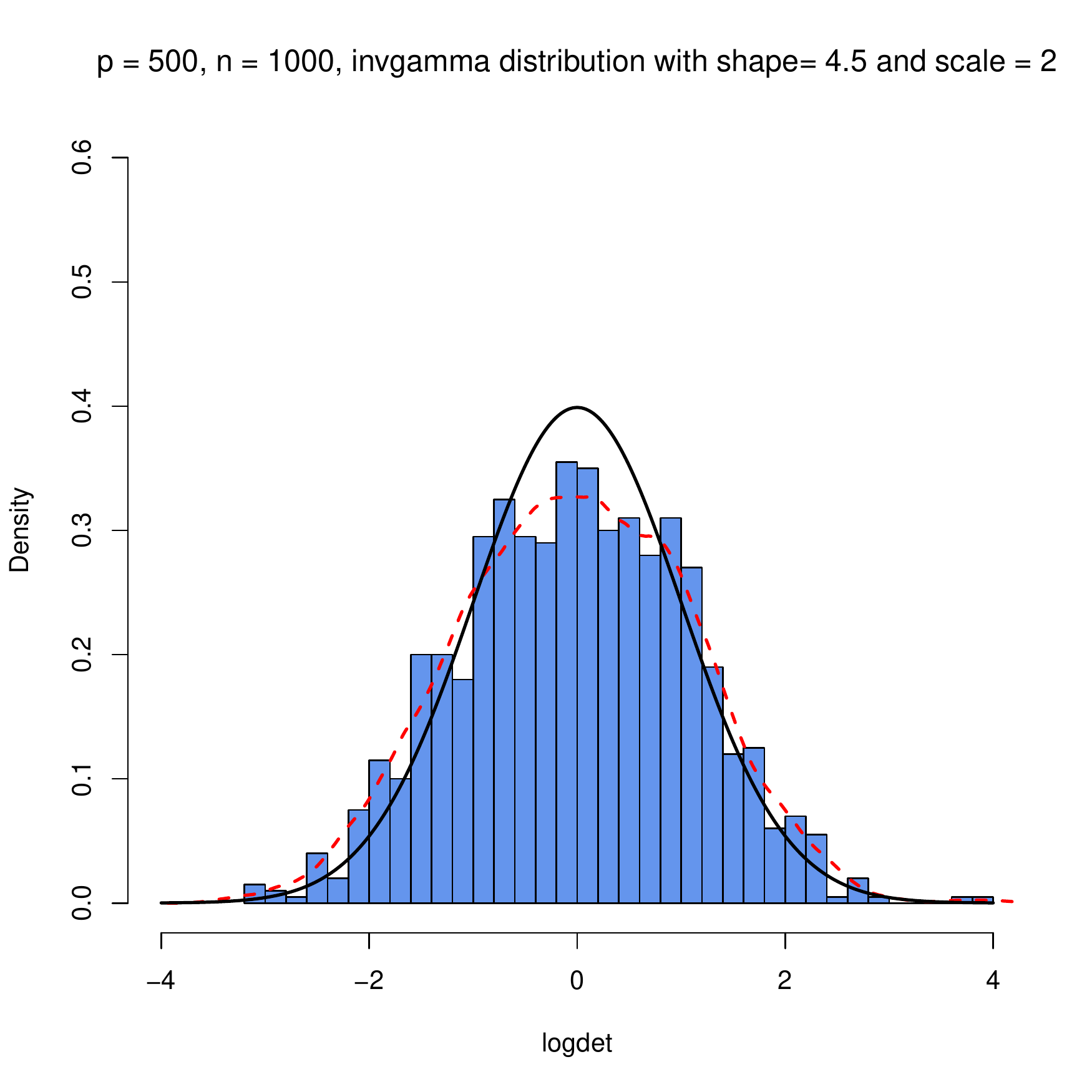}\\
  \includegraphics[scale=0.4]{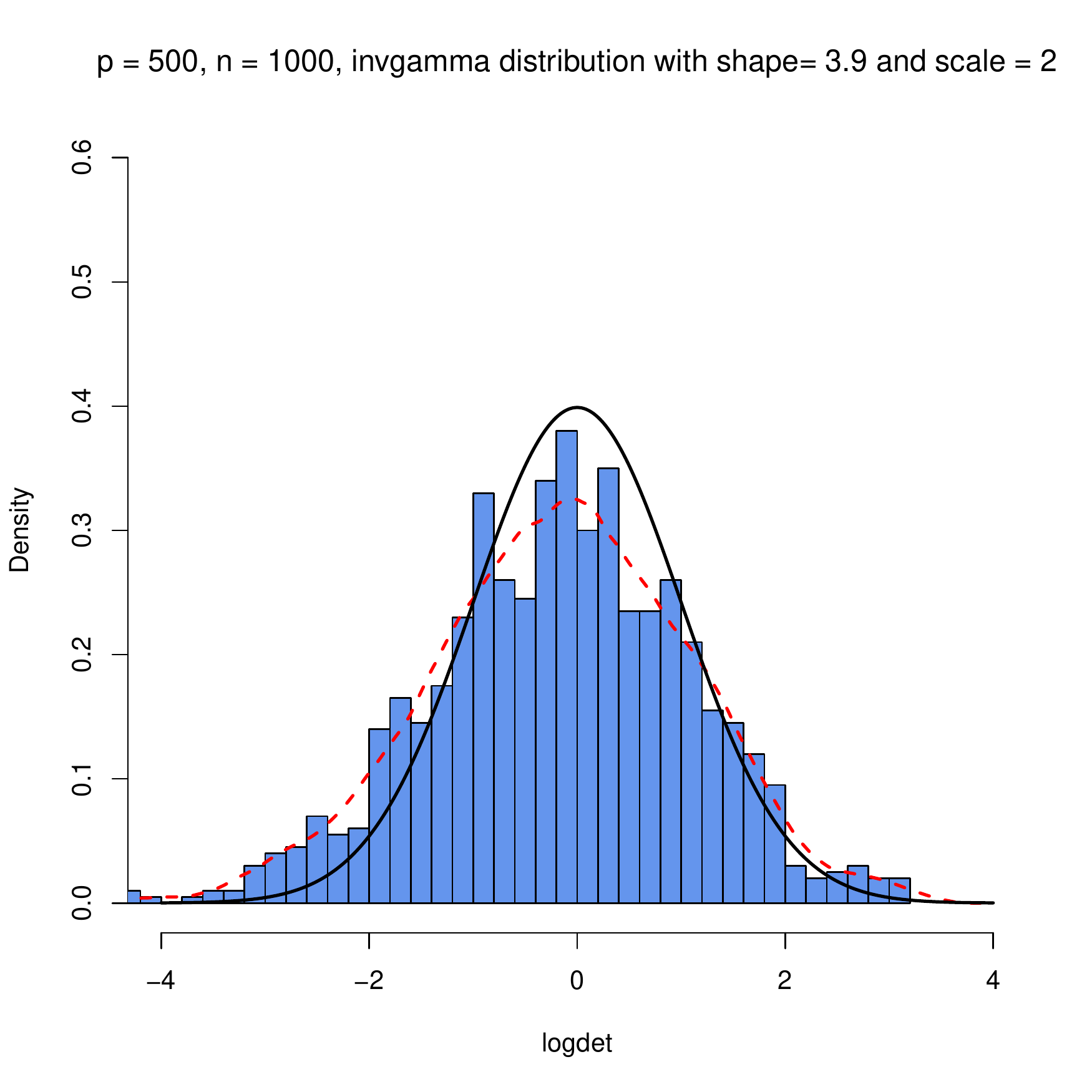}~~
  \includegraphics[scale=0.4]{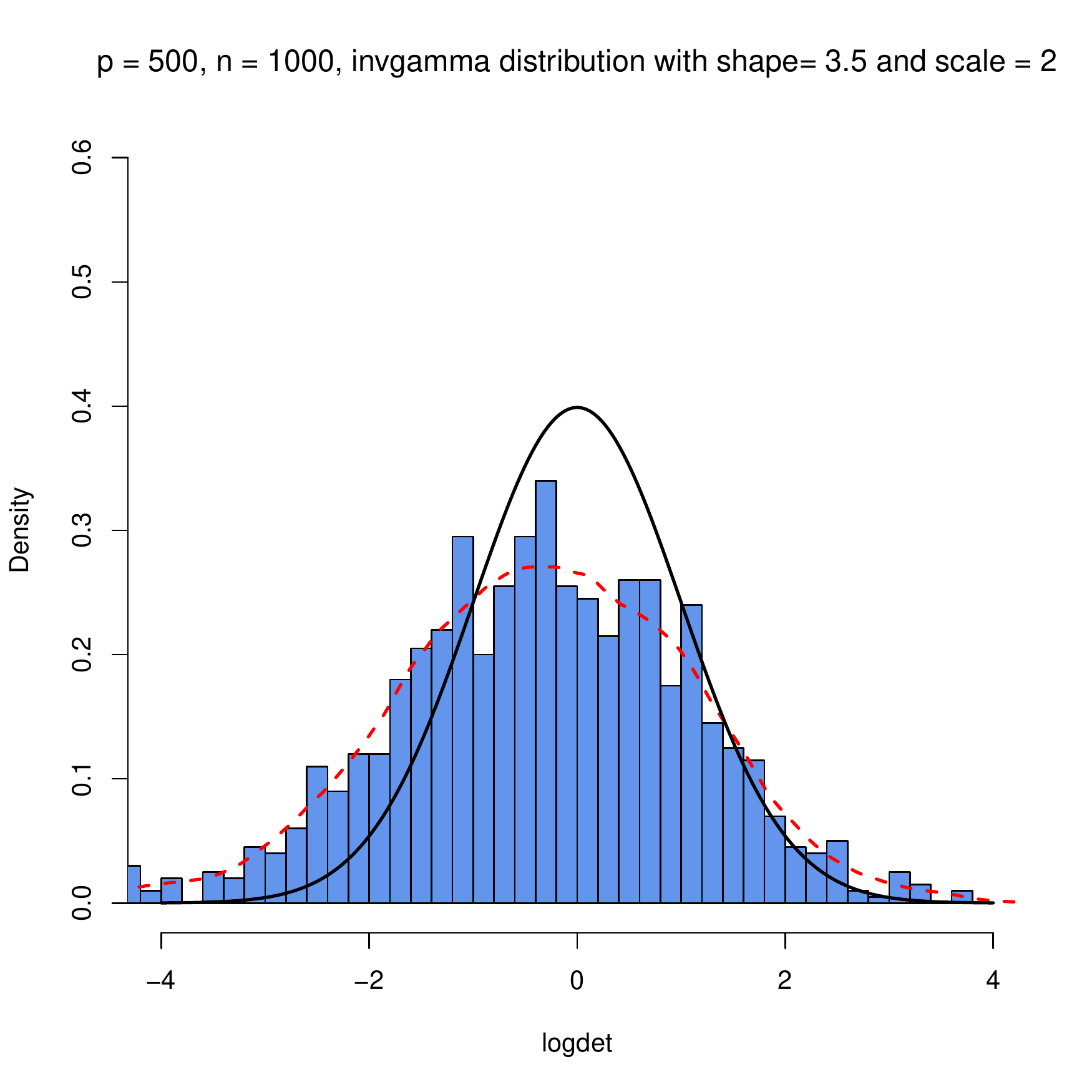}
  \caption{Logarithmic law for inverse gamma distribution with scale $\beta=2$ and shape $\alpha\in \{5,4.5, 3.9, 3.5\}$ for $p=500$ and $n=1000$ with 1000 repetitions.}
  \label{fig2}
\end{figure}

Next, we generate the entries $X_{ij}$ from a non-symmetric distribution, namely inverse gamma with scale parameter $2$ and varying shape parameter. Note that this distribution has a regularly varying tail with index $\alpha$ equal to the shape parameter and the function $L(x)$ from \eqref{eq:regvar} behaving like a constant as $x\to\infty$. Thus, the shape parameter for inverse gamma distribution plays the same role as the degrees of freedom for $t$-distribution, namely if the shape coefficient is smaller than four then the moment of order four does not exist.  Hence, the top row in Figure 2 represents the results when the fourth moment exists, while the pictures in the bottom row represent the case of an infinite fourth moment. One can clearly see that symmetry is vital for logarithmic law to be valid. Indeed, by a careful examination of the proof one can see that asymmetric distribution of $X_{ij}$ as well as a tail parameter $\alpha<3$ could possibly create additional terms in the asymptotic variance and, thus, the CLT in \eqref{eq:main} might not be true anymore.

As a consequence, if our assumptions are violated, the limiting distribution of the logarithmic determinant of the sample correlation matrix still resembles the normal one but with a considerably larger variance. The asymmetry effects seem, however, to have a larger impact on the limiting distribution of $\log\det \bfR$ in the case of heavy tailed data. In case the distribution of heavy-tailed data is not symmetric, it might be beneficial to take an appropriate power transform of the data before using the derived logarithmic law for any testing procedures, for example, testing the uncorrelatedness. \medskip 

Finally, we briefly comment on the extension of our result to $p$-dimensional observations with population covariance $\bfSigma\neq \bfI$, which amounts to replacing the data matrix $\X$ with $\bfSigma^{1/2} \X$, where $\bfSigma^{1/2}$ is the Hermitian square root of $\bfSigma$. In the sample covariance case, since
  \begin{eqnarray*}
   \log \det (\bfSigma^{1/2}\bfS \bfSigma^{1/2}) = \log \det \bfS + \log \det \bfSigma\,,
  \end{eqnarray*}
it is straightforward to obtain a CLT for $\log \det (\bfSigma^{1/2}\bfS \bfSigma^{1/2})$ from \eqref{loglaw:cov}. Unfortunately, there seems to be no such simple relation for the logarithmic determinant of the sample correlation matrix $$\wt \bfR= \{\diag(\bfSigma^{1/2}\bfS \bfSigma^{1/2})\}^{-1/2} \bfSigma^{1/2}\bfS \bfSigma^{1/2}\{\diag(\bfSigma^{1/2}\bfS \bfSigma^{1/2})\}^{-1/2}\,.$$
Recently, \cite{parolya:heiny:kurowicka:2021} used the identity
\begin{eqnarray*}
    \log \det \wt \bfR&=& \log \det (\bfGamma^{1/2} \bfS \bfGamma^{1/2}) - \log \det (\diag(\bfGamma^{1/2} \bfS \bfGamma^{1/2}))\,,
\end{eqnarray*}
where $\bfGamma=\{\diag(\bfSigma)\}^{-1/2}\, \bfSigma\{\diag(\bfSigma)\}^{-1/2}$ is the associated population correlation matrix, to derive a CLT in the case of a finite fourth moment. It is an interesting topic for future research to figure out the dependence on $\bfGamma$ in the heavy-tailed case of infinite fourth moment.

\section{Diagonal part: Exact formula}\label{sec:3}\setcounter{equation}{0}

In this section, we will derive an exact formula for the fourth moment of $\sum_{k=1}^n a_k (nZ_k^2-1)$, where $a_1,\ldots,a_n$ are constants and $Z_1,\ldots, Z_n$ are (essentially) exchangeable random variables satisfying $Z_1^2+\cdots +Z_n^2=1$. We start with the following lemma.
\begin{lemma}\label{lem:fourthmoment}
Let $Z_1,\ldots,Z_n$ be random variables such that, for all positive integers $m_1,\ldots, m_r$ with $m_1+\cdots+m_r\le 4$, $\beta_{2m_1,\ldots, {2m_r}}:=\E[Z_{i_1}^{2m_1}Z_{i_2}^{2m_2} \cdots Z_{i_4}^{2m_4}]$ is finite and invariant under permutations of the indices. 
 Then we have for any numbers $a_1,\ldots,a_n$ with $a_1+\cdots  +a_n=1$ that
	\begin{equation}\label{eq:fourthdiagcentered}
	\begin{split}
	\E&\Big[\Big(\sum_{k=1}^n a_k (Z_k^2-\E[Z_k^2])\Big)^4\Big]
	=S_4\beta_{8}+4(S_3-S_4)\beta_{6,2}-4S_3\beta_{2} \beta_{6}
	+3(S_2^2-S_4)\beta_{4,4}\\
	&+6(S_2-S_2^2-2S_3+2S_4)\beta_{4,2,2}+12(-S_2+S_3)\beta_{2}\beta_{4,2}+4(3S_2-2S_3-1)\beta_{2}\beta_{2,2,2}\\
	&+(-6S_2+3S_2^2+8S_3-6S_4+1)\beta_{2,2,2,2}
	+6(1-S_2)\beta_{2}^2\beta_{2,2}+6S_2\beta_{2}^2\beta_{4}-3\beta_{2}^4\,,
	\end{split}
  \end{equation}
where for $j\ge 1$, we define $S_j=a_1^j+\cdots  +a_n^j$. Moreover, we have
	\begin{equation}\label{eq:fourthdiag}
	\begin{split}
	\E\Big[\Big(\sum_{k=1}^n a_k Z_k^2\Big)^4\Big]&=
	S_4 \beta_8+ 4(S_3-S_4)\beta_{6,2}+
	6(S_2-S_2^2-2S_3+2S_4)\beta_{4,2,2}\\
	&\quad+ 3(S_2^2-S_4)\beta_{4,4}+	(1-6S_2+3S_2^2+8S_3-6S_4)\beta_{2,2,2,2}\,.
	\end{split}
  \end{equation}
\end{lemma}
\begin{proof}
We note that all sums in this proof run from $1$ to $n$. Using $a_1+\cdots  +a_n=1$, it is easy to check that 
\begin{align}
\sum_{k\neq \ell} a_k a_{\ell}&=1-S_2\,, \qquad 
\sum_{k\neq \ell} a_k^2 a_{\ell}=S_2-S_3\,, \qquad 
\sum_{k\neq \ell} a_k^3 a_{\ell}=S_3-S_4\,, \label{eq:fourth1}\\
\sum_{k\neq \ell} a_k^2 a_{\ell}^2&=S_2^2-S_4\,, \qquad 
\sum_{k\neq \ell\neq j} a_k^2 a_{\ell} a_j=S_2-S_2^2-2S_3+2S_4\,, \label{eq:fourth2}\\
\sum_{k\neq \ell\neq j} a_k a_{\ell} a_j&=1-3S_2+2S_3\,, \qquad 
\sum_{k\neq \ell\neq j\neq h} a_k a_{\ell} a_j a_h=1-6S_2 +3S_2^2+8S_3-6 S_4\,.\label{eq:fourth3} 
\end{align}
For example, we shall show the second relation in \eqref{eq:fourth2},
\begin{equation*}
\sum_{k\neq \ell\neq j} a_k^2 a_{\ell} a_j= \sum_{k\neq \ell} a_k^2 a_{\ell}(1-a_k-a_{\ell})=\sum_{k\neq \ell} a_k^2 a_{\ell}-\sum_{k\neq \ell} a_k^3 a_{\ell}-\sum_{k\neq \ell} a_k^2 a_{\ell}^2=S_2-S_2^2-2S_3+2S_4\,.
\end{equation*}
We have the decomposition
\begin{equation}\label{eq:sdgsdsd}
	\begin{split}
	\E\Big[\Big(\sum_{k=1}^n a_k Z_k^2\Big)^4\Big]&=
	\E\Big[\Big(\sum_{k=1}^n a_k^2 Z_k^4 \Big)^2\Big] +
	2\E\Big[\sum_{j=1}^n a_j^2 Z_j^4\sum_{k\neq \ell} a_ka_{\ell} Z_k^2Z_{\ell}^2 \Big]+
	\E\Big[\Big( \sum_{k\neq \ell} a_ka_{\ell} Z_k^2Z_{\ell}^2 \Big)^2\Big]\\
	&=: I+II+III\,.
	\end{split}
  \end{equation}
For the first term, we get
$$I=\beta_8 \sum_{k=1}^n a_k^4+ \beta_{4,4} \sum_{k\neq \ell} a_k^2 a_{\ell}^2=\beta_8 S_4+\beta_{4,4} (S_2^2-S_4)\,,$$
where \eqref{eq:fourth1} was used for the last equality. In view of \eqref{eq:fourth1} and \eqref{eq:fourth2}, we have
\begin{equation*}
\begin{split}
II&=2\beta_{6,2}\sum_{k\neq \ell} a_ka_{\ell}(a_k^2+a_{\ell}^2)+ 2\beta_{4,2,2}\sum_{k\neq \ell} a_ka_{\ell} (S_2-a_k^2-a_{\ell}^2)\\
&= \beta_{6,2} (4S_3-4S_4)+ \beta_{4,2,2}(2 S_2-2 S_2^2 -4S_3+4S_4)\,.
\end{split}
\end{equation*}
Using \eqref{eq:fourth1}--\eqref{eq:fourth3} for the third equality, we find for the third term that
\begin{equation*}
\begin{split}
III&=\beta_{4,4}\,2\sum_{k\neq \ell} a_k^2 a_{\ell}^2+ \beta_{4,2,2}\mathop{\sum_{k\neq \ell,j\neq h}}_{\#\{k,\ell,j,h\}=3} a_k a_{\ell}a_j a_h+ \beta_{2,2,2,2} \sum_{k\neq \ell\neq j\neq h} a_k a_{\ell} a_j a_h\\
&=\beta_{4,4}\,2\sum_{k\neq \ell} a_k^2 a_{\ell}^2+ \beta_{4,2,2}\, 4\sum_{k\neq \ell\neq j} a_k^2 a_{\ell}a_j + \beta_{2,2,2,2} \sum_{k\neq \ell\neq j\neq h} a_k a_{\ell} a_j a_h\\
&= 2\beta_{4,4}(S_2^2-S_4) +4\beta_{4,2,2}(S_2-S_2^2-2S_3+2S_4) +\beta_{2,2,2,2}(1-6S_2+3S_2^2+8S_3-6S_4)\,.
\end{split}
\end{equation*}
Simplifying $I+II+III$ establishes \eqref{eq:fourthdiag} by virtue of \eqref{eq:sdgsdsd}.

Next, we turn to \eqref{eq:fourthdiagcentered}.
To this end, let $\bfA$ be the $n\times n$ diagonal matrix with diagonal entries $a_1,\ldots,a_n$. By Lemma~\ref{lem:quf}, we have
\begin{align}
	\E&\Big[\Big(\sum_{k=1}^n a_k Z_k^2\Big)^3\Big]=  
	\beta_{2,2,2}[(\tr \bfA)^3 +6 \tr \bfA \tr(\bfA^2)+8 \tr(\bfA^3)]+ (\beta_6 -15 \beta_{4,2}+30\beta_{2,2,2}) \tr(\bfA \circ \bfA \circ \bfA ) \nonumber\\
	&\quad +(\beta_{4,2}-3\beta_{2,2,2}) [3\tr \bfA \tr(\bfA \circ \bfA)+12 \tr(\bfA \circ \bfA^2)] \nonumber\\
	&=  \beta_{2,2,2}[1 +6  S_2+8 S_3]+ (\beta_6 -15 \beta_{4,2}+30\beta_{2,2,2}) S_3 +(\beta_{4,2}-3\beta_{2,2,2}) [3 S_2+12 S_3]\,, \label{eq:thirdjuhu}
  \end{align}
where $\circ$ denotes the Hadamard product.
A simple calculation using $a_1+\cdots+a_n=1$ yields
\begin{equation}\label{eq:llff11}
\E\Big[\Big(\sum_{k=1}^n a_k Z_k^2\Big)^2\Big]= \beta_4 S_2 + \beta_{2,2} (1-S_2)\,.
\end{equation}
By the binomial theorem, we have  
\begin{equation}\label{eq:rhdhdff}
\begin{split}
\E&\Big[\Big(\sum_{k=1}^n a_k Z_k^2-\beta_2)\Big)^4\Big]=  \sum_{t=0}^4 \binom{4}{t} \E \Big[ \Big(\sum_{k=1}^n a_kZ_k^2 \Big)^t\Big] (-\beta_2)^{4-t} \,.
\end{split}
\end{equation}
Plugging \eqref{eq:llff11}, \eqref{eq:thirdjuhu} and \eqref{eq:fourthdiag} into \eqref{eq:rhdhdff} and then simplifying establishes \eqref{eq:fourthdiagcentered}. We omit details of this lengthy computation.
\end{proof}

Additionally assuming $Z_1^2+\cdots +Z_n^2=1$, the relation between the $\beta$'s is captured by the following crucial lemma.
\begin{lemma}\label{lem:betas}
Let $Z_1,\ldots,Z_n$ be random variables such that $Z_1^2+\cdots +Z_n^2=1$ and, for all positive integers $m_1,\ldots, m_r$ with $m_1+\cdots+m_r\le 4$, $\beta_{2m_1,\ldots, {2m_r}}:=\E[Z_{i_1}^{2m_1}Z_{i_2}^{2m_2} \cdots Z_{i_4}^{2m_4}]$ is invariant under permutations of the indices.
Then it holds that $\beta_2=1/n$ and
\begin{align}
\beta_4 &= \frac{1}{n}-(n-1)\beta_{2,2}\,,\qquad \beta_{4,2} =\frac{1}{2} \beta_{2,2}-\frac{n-2}{2}\beta_{2,2,2}\,, \label{eq1}\\
\beta_6 &= \frac{1}{n}-\frac{3(n-1)}{2} \beta_{2,2} +\frac{(n-1)(n-2)}{2} \beta_{2,2,2} \,, \label{eq2}\\
\beta_{6,2}&= \frac{1}{2}\beta_{2,2}-\frac{5 (n-2)}{6} \beta_{2,2,2}+\frac{(n-2)(n-3)}{3} \beta_{2,2,2,2} -\beta_{4,4}   \,, \label{eq4}\\
\beta_{4,2,2}&= \frac{1}{3}\beta_{2,2,2}+\frac{3-n}{3} \beta_{2,2,2,2}   \,, \label{eq5}\\
\beta_{8}&= \frac{1}{n}+2(1-n)\beta_{2,2} +\Big(\frac{4 n^2}{3}-4n+\frac{8}{3}\Big) \beta_{2,2,2} \notag \\ 
& \quad +\Big( \frac{-n^3}{3}+2n^2-\frac{11 n}{3} +2 \Big)\beta_{2,2,2,2}+(n-1)\beta_{4,4}.\label{eq3}
\end{align}
\end{lemma}
\begin{proof}
Since $Z_{1}^2+\cdots+Z_{n}^2=1$, an application of the multinomial theorem shows that for $k\ge 1$,
\begin{equation*}
1=(Z_{1}^2+\cdots+Z_{n}^2)^k=
\sum_{r=1}^{k}
 \mathop{\sum_{m_1+\cdots +m_r=k }}_{m_j \ge 1} \binom{n}{r}  \binom{k}{m_1,\ldots, m_r} 
\beta_{2m_1,\ldots, {2m_r}}\,.
\end{equation*}
In particular, for $k=2,3,4$, one obtains
\begin{align}
1&= n \beta_4+n(n-1) \beta_{2,2}\,, \label{id1}\\
1&=n\beta_6 + 3 n(n-1)\beta_{4,2}+n(n-1)(n-2) \beta_{2,2,2}\,, \label{id2}\\
1&=n \beta_8+4 n(n-1) \beta_{6,2}+3n(n-1) \beta_{4,4}+6n(n-1)(n-2) \beta_{4,2,2} \label{id3}\\
&\quad +n(n-1)(n-2)(n-3) \beta_{2,2,2,2} \nonumber\,.
\end{align}
Since $Z_1^2+\cdots+Z_2^2=1$, it holds
$Z_{1}^{2k}=Z_{1}^{2k} (Z_{1}^2+\cdots+Z_{n}^2)$.
Taking expectation one obtains
\begin{equation}\label{eq:trick1}
\beta_{2k}=\beta_{2k+2}+(n-1)\beta_{2k,2} \,,\qquad k=1,2,3\,.
\end{equation}
Using $Z_{1}^{2k}Z_{2}^2=Z_{1}^{2k} Z_{2}^2(Z_{1}^2+\cdots+Z_{n}^2)$, one analogously gets
\begin{equation}\label{eq:trick2}
\beta_{2k,2}=\beta_{2k+2,2}+\beta_{2k,4}+(n-2)\beta_{2k,2,2} \,,\qquad k=1,2\,,
\end{equation}
\begin{equation}\label{eq:trick3}
\beta_{2,2,2}=3\beta_{4,2,2}+(n-3)\beta_{2,2,2,2} \,.
\end{equation}
The lemma now follows from equations \eqref{id1}--\eqref{eq:trick3} and some tedious but straightforward computations.
\end{proof}

We now state the main result of this section.
\begin{theorem}\label{thm:fourthmoment}
Let $Z_1,\ldots,Z_n$ be random variables such that $Z_1^2+\cdots +Z_n^2=1$ and, for all positive integers $m_1,\ldots, m_r$ with $m_1+\cdots+m_r\le 4$, $\beta_{2m_1,\ldots, {2m_r}}:=\E[Z_{i_1}^{2m_1}Z_{i_2}^{2m_2} \cdots Z_{i_4}^{2m_4}]$ is invariant under permutations of the indices.
 Then we have for any numbers $a_1,\ldots,a_n$ with $a_1+\cdots  +a_n=1$ that
	\begin{equation}\label{eq:fourthcomplete}
	\E\Big[\Big(\sum_{k=1}^n a_k (nZ_k^2-1)\Big)^4\Big]
	=K_{4,4}n^4\beta_{4,4}+K_{2,2}n^2\beta_{2,2}+K_{2,2,2}n^3\beta_{2,2,2}+K_{2,2,2,2}n^4\beta_{2,2,2,2}+ K\,,
  \end{equation}
where $S_j=a_1^j+\cdots  +a_n^j$, $j\ge 1$ and
\begin{align*}
K_{4,4}&=3S_2^2-4S_3+nS_4\,,\qquad K=6nS_2-4n^2S_3+ n^3S_4-3\,, \\
K_{2,2}&=-12 nS_2+8n^2 S_3-2n^3S_4+6\,, \\
K_{2,2,2}&=8nS_2-2nS_2^2+\frac{8n(1-2n)}{3}S_3+\frac{2n^2(2n-1)}{3}S_4-4\,, \\
K_{2,2,2,2}&=-2nS_2+(2n-3)S_2^2+\frac{4(n^2-2n+3)}{3}S_3-\frac{n(n^2-2n+3)}{3}S_4+1\,.
\end{align*}
In particular, we have
\begin{equation}\label{eq:sumK}
K+K_{4,4}+K_{2,2}+K_{2,2,2}+K_{2,2,2,2}=0\,.
\end{equation}
\end{theorem}
\begin{proof}
We have 
$$\E\Big[\Big(\sum_{k=1}^n a_k (nZ_k^2-1)\Big)^4\Big]=n^4 \E\Big[\Big(\sum_{k=1}^n a_k (Z_{k}^2-1/n)\Big)^4\Big]\,.$$
The \rhs~ can be explicitly computed using Lemma~\ref{lem:fourthmoment}. Plugging in the formulas from Lemma~\ref{lem:betas}, one can check, for example with mathematical software, that \eqref{eq:fourthcomplete} holds.

Even though, equation \eqref{eq:sumK} follows from the defintions of $ K,K_{4,4},K_{2,2},K_{2,2,2},K_{2,2,2,2}$. We will provide an additional proof which is more insightful. To this end, set $Z_1=\cdots=Z_n=n^{-1/2}$ which implies that the \lhs~in \eqref{eq:fourthcomplete} is zero and that the \rhs~is $K+K_{4,4}+K_{2,2}+K_{2,2,2}+K_{2,2,2,2}$.
\end{proof}
While the main focus of this paper is on the sample correlation matrix $\bfR=\Y\Y^{\top}$ (see \eqref{eq:defRY}), Theorem~\ref{thm:fourthmoment} might be of independent interest. We will apply Theorem~\ref{thm:fourthmoment} to the rows of $\Y$. Since $Y_{11},\ldots, Y_{1n}$ are exchangeable random variables satisfying $Y_{11}^2+\cdots+Y_{1n}^2=1$, one obtains the following corollary.
\begin{corollary}\label{cor:fourthmoment}
Let $Y_{11},\ldots,Y_{1n}$ be defined in \eqref{def:R} and for all positive integers $m_1,\ldots, m_r$ with $m_1+\cdots+m_r\le 4$ set $\beta_{2m_1,\ldots, {2m_r}}:=\E[Y_{11}^{2m_1}Y_{12}^{2m_2} \cdots Y_{1r}^{2m_4}]$.
 Then we have for any numbers $a_1,\ldots,a_n$ with $a_1+\cdots  +a_n=1$ that
	\begin{equation}\label{eq:corfourthcomplete}
	\E\Big[\Big(\sum_{k=1}^n a_k (nY_{1k}^2-1)\Big)^4\Big]
	=K_{4,4}n^4\beta_{4,4}+K_{2,2}n^2\beta_{2,2}+K_{2,2,2}n^3\beta_{2,2,2}+K_{2,2,2,2}n^4\beta_{2,2,2,2}+ K\,,
  \end{equation}
where $S_j=a_1^j+\cdots  +a_n^j$, $j\ge 1$, and $K_{4,4},K_{2,2},K_{2,2,2},K_{2,2,2,2},K$ are defined in Theorem~\ref{thm:fourthmoment}.
\end{corollary}

\section{Proof of the main result}\setcounter{equation}{0}\label{sec:proofmain}

\subsection{Preliminaries}

Throughout this section, for integers $k_1,\ldots,k_r$, we will use the notation 
$$\beta_{2k_1,\ldots, {2k_r}}:=\E [ Y_{11}^{2k_1} \cdots Y_{1r}^{2k_r} ]\,,$$
where we recall the definition of $Y_{ij}$ from \eqref{def:R}.
Since $\beta_{2k_1,\ldots, {2k_r}}=\beta_{2k_{\pi(1)},\ldots, 2k_{\pi(r)}}$ for any permutation $\pi$ on $\{1,\ldots,r\}$ we will typically write the indices in decreasing order. For example, instead of $\beta_{2,4}$ we prefer writing $\beta_{4,2}$. 
Now we compute the precise asymptotic behavior of $\beta_{2k_1,\ldots, {2k_r}}$.

\begin{lemma}\label{lem:moment24}
Let $\alpha\in (2,4)$ and assume that $\E[X_{11}^2]=1$ and $\P(|X_{11}|>x)=x^{-\alpha} L(x)$ for $x>0$ where $L$ is a slowly varying function. Define the $Y_{kn}$'s as in \eqref{def:R} and consider integers $k_1,\ldots,k_r\ge 1$. Then it holds
\begin{equation}\label{moment24}
\lim_{\nto} \frac{n^{N_1(1-\alpha/2)+ r\alpha/2}}{L^{r-N_1}(n^{1/2})}  \beta_{2k_1,\ldots, {2k_r}} =   \frac{(\alpha/2)^{r-N_1}\Gamma(N_1(1-\alpha/2)+ r\alpha/2) \, \prod_{i:k_i\ge 2} \Gamma(k_i-\alpha/2)}{\Gamma(k_1+\cdots+k_r)}\,,
\end{equation}
where $N_1=\#\{1\le i\le r: k_i=1\}$.
In particular, we have
\begin{equation}\label{highestmoment24}
\lim_{\nto}\frac{n^{\alpha/2}}{L(n^{1/2})}\beta_{2k} = \frac{\alpha \Gamma(\alpha/2)\Gamma(k-\alpha/2)}{2 \Gamma(k)}\,, \qquad k\ge 1\,.
\end{equation}
\end{lemma}
\begin{proof}
We remark that \eqref{moment24} was proved in \cite{albrecher:teugels:2007} for $N_1=0$, that is $k_i\ge 2$. For the general case let $\beta=\alpha/2$, $X\eid X_{11}$ and consider $r\ge 1$, $k_1+\cdots+k_r=k\ge 1$ with $k_i\ge 1$. From Albrecher and Teugels \cite{albrecher:teugels:2007}, page 7, we have
\begin{equation}\label{eq:formulagine}
\E[ Y_{11}^{2k_1}\cdots Y_{1r}^{2k_r}]= \frac{(-1)^k}{n \Gamma(k)} \int_0^{\infty} \Big( \tfrac{t}{n}\Big)^{k-1} \varphi^{n-r}\Big( \tfrac{t}{n}\Big) \prod_{i=1}^r \varphi^{(k_i)}\Big( \tfrac{t}{n}\Big) \dint t\,,
\end{equation}
where $\varphi(s)=\E[\e^{-sX^2}]$, $s>0$, and $\varphi^{(m)}(s)=\frac{\dint^m}{\dint s^m} \varphi(s)$.
By \cite{albrecher:teugels:2007}, we have 
\begin{equation}\label{lim1}
\lim_{\nto} \varphi^{n-r}\Big( \tfrac{t}{n}\Big)=\e^{-t}\,, \qquad t>0\,,
\end{equation}
and that the asymptotic behavior of $\varphi^{(m)}(s)$, $m\in \N$, at the origin is given by
\begin{eqnarray}\label{eq:asyphi}
(-1)^m \varphi^{(m)}(s) \sim
\left\{\begin{array}{ll}
\beta \Gamma(m-\beta) s^{\beta -m} L(s^{-1/2}) \,, & \mbox{if } m>\beta, \\
\E[X^{2m}] \,, & \mbox{if }  m\le \beta,
\end{array}\right. \qquad s \downarrow 0\,.
\end{eqnarray}

By \eqref{eq:formulagine}, Potter's theorem and the dominated convergence theorem (for more details see \cite{albrecher:teugels:2007} or \cite{fuchs:joffe:teugels:2001}), we obtain in view of \eqref{lim1} and \eqref{eq:asyphi} that, as $\nto$,
\begin{equation*}
\begin{split}
\E[ Y_{11}^{2k_1}\cdots Y_{1r}^{2k_r}]&= \frac{(-1)^k}{n \Gamma(k)} \int_0^{\infty} \Big( \tfrac{t}{n}\Big)^{k-1} \varphi^{n-r}\Big( \tfrac{t}{n}\Big) \Big(\varphi^{(1)}\big( \tfrac{t}{n}\big)\Big)^{N_1} \prod_{i:k_1\ge 2} \varphi^{(k_i)}\Big( \tfrac{t}{n}\Big) \dint t\\
&\sim \frac{1}{n \Gamma(k)} \int_0^{\infty} \Big( \tfrac{t}{n}\Big)^{k-1} \e^{-t} \Big(\E[X^{2}]\Big)^{N_1} \prod_{i:k_i\ge 2} \beta \Gamma(k_i-\beta) \big( \tfrac{t}{n}\big)^{\beta -k_i} \underbrace{L\Big(\big( \tfrac{t}{n}\big)^{-1/2}\Big)}_{\sim L(n^{1/2})} \dint t\\
&\sim \Big(\prod_{i:k_i\ge 2} \Gamma(k_i-\beta) \Big) \frac{\beta^{r-N_1}L^{r-N_1}(n^{1/2})}{n^{N_1(1-\beta)+\beta r} \Gamma(k)} \int_0^{\infty} \e^{-t} t^{N_1(1-\beta)+\beta r-1} \dint t\\
&= \frac{L^{r-N_1}(n^{1/2})}{n^{N_1(1-\beta)+\beta r}} \frac{\beta^{r-N_1}\Gamma(N_1(1-\beta)+\beta r) \, \prod_{i:k_i\ge 2} \Gamma(k_i-\beta)}{\Gamma(k)}\,.
\end{split}
\end{equation*}
Rearranging yields \eqref{moment24}.
\end{proof}

\begin{remark}{\em
We mention that \eqref{moment24} per se does not tell us the speed of convergence of the \lhs~to the limit. For example, by  \eqref{moment24} we (only) know that $n(n-1)\beta_{2,2}\sim 1$, as $\nto$. Using the first identity in \eqref{eq1}, we deduce that
$$1-n(n-1)\beta_{2,2}=n \beta_{4}\sim n^{1-\alpha/2}L(n^{1/2}) (\alpha/2)\, \Gamma(\alpha/2)\Gamma(2-\alpha/2)\,,\quad \nto\,,$$
where \eqref{highestmoment24} was used in the last step. Thus, for certain cases, Lemma~\ref{lem:moment24} in conjunction with Lemma~\ref{lem:betas} reveal the speed of convergence in \eqref{moment24}.
}\end{remark}

\subsection{Proof of Theorem~\ref{thm:main}}
With some matrix algebra, Wang et al.~\cite[p.~85-86]{WangHanPan2018} derived for the log determinant of the sample covariance matrix $\bfS=n^{-1}\X\X^{\top}$ that
\begin{equation}\label{eq:segtse}
\log \det \bfS= -p \log n + \log((n(n-1)\cdots (n-p+1))+ \sum_{i=0}^{p-1} \log(1+Z_{i+1})\,,
\end{equation}
where
\begin{equation*}
Z_{i+1}=\frac{b_{i+1}^{\top} P_i b_{i+1} -(n-i)}{n-i}\quad \text{ and } \quad P_i=\bfI_n-B_{(i)}^{\top} (B_{(i)} B_{(i)}^\top)^{-1} B_{(i)}\,.
\end{equation*}
Here $P_0=\bfI_n$, $B_{(i)}=(b_1,\ldots, b_i)^{\top}$, and $b_i=(x_{i1}, \ldots,x_{in})^{\top}$ denotes the $i$th row of the matrix $\X$, $i=1,\ldots, p-1$.

Analogously to \eqref{eq:segtse}, we get for the log determinant of the sample correlation matrix $\bfR=\Y\Y^{\top}$ that
\begin{equation}\label{eq:segtse1}
\log \det \bfR= \underbrace{-p \log n + \log(n(n-1)\cdots (n-p+1))}_{:=c_n}+ \sum_{i=0}^{p-1} \log(1+\wt Z_{i+1})\,,
\end{equation}
where
\begin{equation*}
\wt Z_{i+1}=\frac{n\,\wt b_{i+1}^{\top} \wt P_i \wt b_{i+1} -(n-i)}{n-i}\quad \text{ and } \quad \wt P_i=\bfI_n-\wt B_{(i)}^{\top} (\wt B_{(i)} \wt B_{(i)}^\top)^{-1} \wt B_{(i)}\,.
\end{equation*}
Here $\wt P_0=\bfI_n$, $\wt B_{(i)}=(\wt b_1,\ldots, \wt b_i)^{\top}$, and $\wt b_i=(Y_{i1}, \ldots,Y_{in})^{\top}$ denotes the $i$th row of the matrix $\Y$. 
An important observation is that
$$\wt P_i=P_i =\bfI_n-B_{(i)}^{\top} (B_{(i)} B_{(i)}^\top)^{-1} B_{(i)}$$
is the same projection matrix as in the sample covariance case. Moreover, due to \cite[Proposition 2.1]{WangHanPan2018} all matrices $B_{(i)} B_{(i)}^\top$ are invertible with overwhelming probability.

We note that $P_i=P_i^2$ and $\tr(P_i)=n-i$, and define $$Q_i=(q_{i,kl})=P_i/(n-i)\,, \qquad 0\le i\le p-1\,.$$
By \cite[Lemma 2.1]{mohammadi:2016} and \cite[Lemma 3.1]{mohammadi:2016}, we have for $0\le i\le p-1$ and $1\le k,l\le n$ that
\begin{equation}\label{eq:boundelements}
0\le q_{i,kk} \le \frac{1}{n-i} \quad \text{and} \quad -\frac{1}{2(n-i)} \le q_{i,kl}\le \frac{1}{2(n-i)}\,.
\end{equation}
It is convenient to decompose $\wt Z_{i+1} $ as follows,
\begin{equation}\label{eq:decZ}
\wt Z_{i+1}  =\sum_{j=1}^n q_{i,jj} (nY_{i+1,j}^2 -1)+ \sum_{k\neq l} q_{i,kl} \,nY_{i+1,k} Y_{i+1,l}=:U_{i+1}+V_{i+1}\,,\qquad 0\le i\le p-1.
\end{equation}
The following result is the key ingredient; it will be proved in Section \ref{sec:proofofprop:important}.
\begin{proposition}\label{prop:important}
In the setting of Theorem~\ref{thm:main}, if $\alpha\in (2,4)$\footnote{We emphasize that some parts of our proof also work for $\alpha>2$, which is the widest range of the tail parameter $\alpha$ for which the CLT for the log-determinant might hold. This is due to fact that for $\alpha\in (0,2)$ the limiting spectral distribution of the sample correlation matrix $\bfR$ is no longer the classical \MP law but the so-called $\alpha$-heavy \MP law; see \cite{heiny:yao:2021} for details.}, we have for any $\vep\in(0,\alpha/2-1)$ that
\begin{equation}\label{eq:fourthV}
\lim_{\nto} n^{\vep} \sum_{i=0}^{p-1} \E[V_{i+1}^4] =0\,.
\end{equation}
Moreover, if $\alpha\in (3,4)$, there exists an $\vep>0$ such that
\begin{equation}\label{eq:fourthU}
\lim_{\nto} n^{\vep} \sum_{i=0}^{p-1} \E[U_{i+1}^4] =0\,.
\end{equation}
\end{proposition}

By Taylor's theorem, we get
\begin{equation}\label{eq:taylornew}
\sum_{i=0}^{p-1} \log(1+\wt Z_{i+1})= \sum_{i=0}^{p-1} (\wt Z_{i+1}-\frac{\wt Z_{i+1}^2}{2}) +\sum_{i=0}^{p-1} R_{i+1}\,,
\end{equation}
where the remainder in Lagrange form is given by
\begin{equation}\label{eq:Ri+1}
R_{i+1}=\frac{1}{3} \Big(\frac{\wt Z_{i+1}}{1+\theta \wt Z_{i+1}} \Big)^3 \quad \text{ for some } \theta=\theta(\wt Z_{i+1})\in (0,1)\,.
\end{equation}
This expansion is justified by 
\begin{equation}\label{eq:maxZ}
\max_{i=0,\ldots,p-1} |\wt Z_{i+1}| \cip 0\,, \qquad \nto\,,
\end{equation}
which is an immediate consequence of the following lemma.
\begin{lemma}\label{lem:4.2}
Under the conditions of Theorem~\ref{thm:main}, we have for any $\vep >0$ that
$$\lim_{\nto} \sum_{i=0}^{p-1} \P(|\wt Z_{i+1}|>\vep) = 0\,.$$
\end{lemma}
\begin{proof}
Using Markov's inequality, $|a+b|^4\le 2^{3} (|a|^4+|b|^4)$ for $a,b\in \R$, and Proposition~\ref{prop:important}, we get for $\vep>0$,
\begin{equation}\label{eq:sumz4}
\sum_{i=0}^{p-1} \P(|\wt Z_{i+1}|>\vep)\le \sum_{i=0}^{p-1}\frac{\E[\wt Z_{i+1}^4]}{\vep}\le \frac{8}{\vep} \sum_{i=0}^{p-1} (\E[U_{i+1}^4]+ \E[U_{i+1}^4])\to 0 \,, \qquad \nto\,.
\end{equation}
\end{proof}
Let $\mathcal{F}_k=\mathcal{F}_k^{(n)}$ be the sigma algebra generated by the first $k$ rows of $\X$.
We have
\begin{equation}\label{eq:decomposecor}
\begin{split}
\sum_{i=0}^{p-1} (\wt Z_{i+1}-\frac{\wt Z_{i+1}^2}{2})&=\sum_{i=0}^{p-1} \wt Z_{i+1}-\sum_{i=0}^{p-1} \underbrace{\tfrac{1}{2} (\wt Z_{i+1}^2-\E[\wt Z_{i+1}^2 | \mathcal{F}_i])}_{:=\wt Y_{i+1}} - \sum_{i=0}^{p-1}\tfrac{1}{2}\E[\wt Z_{i+1}^2 | \mathcal{F}_i]\,.
\end{split}
\end{equation}
 Define $\mu_n=(p-n+\frac{1}{2})\log(1-\frac{p}{n})-p +\frac{p}{n}$, (which is the centering sequence in the CLT).

In view of \eqref{eq:segtse1} and \eqref{eq:taylornew}, one gets
\begin{equation}\label{eq:dddd}
\log \det \bfR -\mu_n = 
\sum_{i=0}^{p-1} \wt Z_{i+1}-\sum_{i=0}^{p-1} \wt Y_{i+1}+\sum_{i=0}^{p-1} R_{i+1}
  - \sum_{i=0}^{p-1}\tfrac{1}{2}\E[\wt Z_{i+1}^2 | \mathcal{F}_i]+c_n-\mu_n\,.
\end{equation}
By virtue of \eqref{eq:dddd} and noting that $-2\log(1-p/n)- 2 p/n\to -2\log(1-\gamma)-2\gamma$, Theorem~\ref{thm:main} follows from the next four limit relations by an application of the Slutsky lemma,

\begin{align}
\frac{1}{\sqrt{-2\log(1-p/n)- 2 p/n}} \sum_{i=0}^{p-1} \wt Z_{i+1} &\cid N(0,1)\,, \label{sumZ}\\
\sum_{i=0}^{p-1} \wt Y_{i+1}&\cip 0\,, \label{sumY_i}\\
\sum_{i=0}^{p-1} R_{i+1}&\cip 0\,, \label{sumR_i}\\
\sum_{i=0}^{p-1} \tfrac{1}{2}\E[\wt Z_{i+1}^2 | \mathcal{F}_i]-c_n+\mu_n&\cip 0\,. \label{sumconst}
\end{align}
Equations \eqref{sumZ}, \eqref{sumY_i}, \eqref{sumR_i}, \eqref{sumconst} are proved in Sections \ref{sec:sumZ}, \ref{sec:sumY_i}, \ref{sec:sumR_i} and \ref{sec:sumconst}, respectively. This completes the proof of Theorem \ref{thm:main}.

\subsection{Proof of \eqref{sumZ}}\label{sec:sumZ}

We will use the following CLT for martingale differences.
\begin{lemma}[e.g. Hall and Heyde \cite{hall:heyde:1980}]\label{lem:martingaleclt}
Let $\{S_{ni},\mathcal{F}_{ni}, 1\le i\le k_n, n\ge 1\}$ be a zero-mean, square integrable martingale array with differences $Z_{ni}$. Suppose that $\E[\max_i Z_{ni}^2]$ is bounded in $n$ and that
 \begin{equation*}
\max_i |Z_{ni}|\cip 0 \quad \text{ and } \quad \sum_i Z_{ni}^2 \cip 1\,.
\end{equation*}
Then we have $S_{nk_n}\cid N(0,1)$.
\end{lemma}
In view of $\E[\wt Z_{i+1} | \mathcal{F}_i]=0$, we observe that $(\wt Z_{i+1})_i$ is a martingale difference sequence with respect to the filtration $(\mathcal{F}_i)$.
We apply Lemma~\ref{lem:martingaleclt} to the martingale differences $\sigma_n \wt Z_{i+1}$ with $\sigma_n=(-2\log(1-p/n)- 2 p/n)^{-1/2}$.
From \eqref{eq:maxZ}, we have $\max_{i=0,\ldots,p-1} | \sigma_n \,\wt Z_{i+1}| \cip 0$ as $\nto$.
In order to check the other conditions in Lemma~\ref{lem:martingaleclt}, we need the following lemmas. The notation 
$S_j^{(i)}:=q_{i,11}^j+\cdots  +q_{i,nn}^j$, $j\ge 1$ will be useful.
\begin{lemma}\label{lem:secondmoment}
Assume that the distribution of $X_{11}$ is symmetric, i.e., $X_{11} \eid -X_{11}$. Then it holds for $0\le i\le p-1$ that
\begin{equation}\label{eq:s2}
 \E[U_{i+1}^2]
= \frac{1-n\E[S_2^{(i)}]}{n-1} (1-n^2 \beta_4)\,,
\end{equation}
\begin{equation}\label{eq:V2}
  \E[V_{i+1}^2]=2 n^2 \beta_{2,2} \,\Big(\frac{1}{n-i}-\E[S_2^{(i)}]\Big)\,.
\end{equation}
\end{lemma}
\begin{proof}
Let $0\le i\le p-1$. By the binomial theorem, we have  for $s\ge 1$,
\begin{equation}\label{eq:rhdhddff}
\begin{split}
\E& \Big[\Big( n\sum_{k=1}^n q_{i,kk} Y_{i+1,k}^2 -1\Big)^s \,\Big| \, \mathcal{F}_i \Big]=(-1)^s +s(-1)^{s-1}  +\sum_{t=2}^s \binom{s}{t} (-1)^{s-t} n^t \E \Big[ \Big(\sum_{k=1}^n q_{i,kk} Y_{i+1,k}^2 \Big)^t \,\Big| \, \mathcal{F}_i \Big]\,.
\end{split}
\end{equation}
A simple calculation using $\tr(Q_i)=1$ yields
\begin{equation}\label{eq:llff}
n^2 \E \Big[ \Big(\sum_{k=1}^n q_{i,kk} Y_{i+1,k}^2 \Big)^2 \,\Big| \, \mathcal{F}_i \Big]=n^2 \beta_4 S_2^{(i)} +n^2 \beta_{2,2} (1-S_2^{(i)})\,.
\end{equation}
Combining \eqref{eq1} and \eqref{eq:llff}, we obtain
\begin{equation*}
n^2 \E \Big[ \Big(\sum_{k=1}^n q_{i,kk} Y_{i+1,k}^2 \Big)^2 \,\Big| \, \mathcal{F}_i \Big]=n^2 \beta_4 \Big(S_2^{(i)} \Big(1+\frac{1}{n-1}\Big)- \frac{1}{n-1}\Big) + \frac{n}{n-1}(1-S_2^{(i)})\,.
\end{equation*}
In view of \eqref{eq:rhdhddff}, this establishes \eqref{eq:s2}.

By conditioning on $\mathcal{F}_i$ and using that $q_{i,kl}=q_{i,lk}$, one gets that
\begin{equation*}
  \E[V_{i+1}^2]=2 n^2 \beta_{2,2} \, \E\sum_{k\neq l} q_{i,kl}^2=2 n^2 \beta_{2,2} \,\Big(\frac{1}{n-i}-\E[S_2^{(i)}]\Big)\,,
\end{equation*}
where we used $\sum_{k,l} q_{i,kl}^2=(n-i)^{-1}$ in the last step.
\end{proof}

\begin{lemma}\label{lem:asymptoticvariance}
Under the assumptions of Theorem \ref{thm:main}, it holds that, as $\nto$,
\begin{equation*}
\sum_{i=0}^{p-1}\E[U_{i+1}^2]= O\big(n^{(3-\alpha)/2+\vep}\big) \qquad \text{and} \qquad 
\sum_{i=0}^{p-1}\E[V_{i+1}^2]\sim -2\log(1-\tfrac pn)- 2 \tfrac pn
\end{equation*}
for any $\vep >0$.
\end{lemma}
\begin{proof}
From Lemma~\ref{lem:secondmoment}, equation \eqref{highestmoment24} and an application of Lemma~\ref{lem:yaskov}, we get for any $\vep >0$,
\begin{align*}
\sum_{i=0}^{p-1}\E[U_{i+1}^2]&= \sum_{i=0}^{p-1} \frac{1-n\E[S_2^{(i)}]}{n-1} (1-n^2 \beta_4) \le n^2 \beta_4 \sum_{i=0}^{p-1} \frac{n\E[S_2^{(i)}]-1}{n-1}\\
&= n^2 \beta_4 \left( \frac{-p}{n-1}+ \frac{n}{n-1} \Big(\sum_{i=0}^{p-1}\E[S_2^{(i)}] - \frac{p}{n} \Big) +\frac{p}{n-1} \right)\\
&=n^{2-\alpha/2+\vep} O(n^{-1/2})= O\big(n^{(3-\alpha)/2+\vep}\big)\,.
\end{align*}
Again from Lemma~\ref{lem:secondmoment} and Lemma~\ref{lem:yaskov}, we get, as $\nto$,
\begin{align*}
\sum_{i=0}^{p-1}\E[V_{i+1}^2]&= 2 n^2 \beta_{2,2} \,\left(\sum_{i=0}^{p-1}\frac{1}{n-i}-\Big(\sum_{i=0}^{p-1}\E[S_2^{(i)}] - \frac{p}{n} \Big) - \frac{p}{n} \right)\\
&=2 \underbrace{n^2 \beta_{2,2}}_{\sim 1} \,\left(\sum_{i=0}^{p-1}\frac{1}{n-i}-O(n^{-1/2}) - \frac{p}{n} \right)\\
&\sim -2\log(1-p/n)- 2 p/n
\end{align*}
since $\sum_{i=0}^{p-1}\frac{1}{n-i} \sim -\log (1-p/n)$.
\end{proof}

Recalling the definition of $\sigma_n^2$ and using Lemma~\ref{lem:asymptoticvariance}, we see that
\begin{equation}\label{eq:variance}
\sigma_n^2 \E\Big[\max_{i=0,\ldots,p-1} \wt Z_{i+1}^2\Big]\le \sigma_n^2 \sum_{i=0}^{p-1} \E[\wt Z_{i+1}^2] = \sigma_n^2 \sum_{i=0}^{p-1} \Big( \E[U_{i+1}^2]+\E[V_{i+1}^2]\Big)=1+o(1)\,.
\end{equation}

Due to $\sigma_n^2 \sum_{i=0}^{p-1}\E[\wt Z_{i+1}^2]=1+o(1)$, the condition $\sigma_n^2 \sum_{i=0}^{p-1}\wt Z_{i+1}^2\cip 1$ is implied by 
\begin{equation}\label{eq:1212}
 \sum_{i=0}^{p-1}(\wt Z_{i+1}^2-\E[\wt Z_{i+1}^2\,| \mathcal{F}_i])\cip 0 \,,\qquad \nto\,,
\end{equation}
and 
\begin{equation}\label{eq:reess}
 \sum_{i=0}^{p-1}(\E[\wt Z_{i+1}^2\,| \mathcal{F}_i]-\E[\wt Z_{i+1}^2]) \cip 0 \,,\qquad \nto\,.
\end{equation}
Observe that \eqref{eq:1212} is equivalent to \eqref{sumY_i}. Hence, it remains to show \eqref{eq:reess}. To this end, recall that in Lemma \ref{lem:secondmoment} and its proof it was calculated that
\begin{align*}
\sum_{i=0}^{p-1}\Big(\E[\wt Z_{i+1}^2\,| \mathcal{F}_i]-\E[\wt Z_{i+1}^2]\Big)&=\sum_{i=0}^{p-1}\Big(\E[U_{i+1}^2 +V_{i+1}^2| \mathcal{F}_i]-\E[U_{i+1}^2 +V_{i+1}^2]\Big)\\
&=\sum_{i=0}^{p-1}\frac{n (S_2^{(i)}-\E[S_2^{(i)}])}{n-1} (1-n^2 \beta_4) +2 n^2 \beta_{2,2} \,\big(S_2^{(i)}-\E[S_2^{(i)}]\big)\\
&\sim (3-\underbrace{n^2 \beta_4}_{\le n^{2-\alpha/2 +\vep}}) \sum_{i=0}^{p-1}\big(S_2^{(i)}-\E[S_2^{(i)}]\big)\cip 0\,, \qquad \nto\,,
\end{align*}
where we used Lemma~\ref{lem:moment24} for the inequality in the last line, and Lemma~\ref{lem:yaskov} in the last step. Indeed, using Lemma \ref{lem:yaskov} for $k=2$ we obtain
\begin{eqnarray*}
  \sum_{i=0}^{p-1}\big(S_2^{(i)}-\E[S_2^{(i)}]\big) &=&\sum_{i=0}^{p-1}\big(S_2^{(i)}-\frac{1}{n}\big)+\sum_{i=0}^{p-1}\big(\E[S_2^{(i)}]-\frac{1}{n}\big)\\
&=&\underbrace{\sum_{i=0}^{p-1}\sum\limits_{\ell=1}^n\left(q^2_{i,\ell\ell}-\frac{1}{n^2}\right)}_{=O_{\P}(n^{-1/2}), ~\text{Markov and Lemma \ref{lem:yaskov}}} +~~~~ \underbrace{\sum_{i=0}^{p-1}\sum\limits_{\ell=1}^n\left(\E[q^2_{i,\ell\ell}]-\frac{1}{n^2}\right)}_{=O(n^{-1/2}), ~\text{Lemma \ref{lem:yaskov}}}  =O_{\P}(n^{-1/2}) \,,
\end{eqnarray*}
where for the first sum we have also used the fact that $0\leq\sum_{\ell=1}^n\left(q^2_{i,\ell\ell}-\frac{1}{n^2}\right)$ by \eqref{eq:sdfsdd}.

Thus, we have verified the conditions of Lemma~\ref{lem:martingaleclt} which now yields \eqref{sumZ} and finishes the proof.

\subsection{Proof of \eqref{sumY_i}}\label{sec:sumY_i}

By Markov's inequality, one has for $\vep>0$,
\begin{equation}\label{eq:boundY}
\P\Big(\Big|\sum_{i=0}^{p-1} \wt Y_{i+1}\Big|>\vep\Big)\le \vep^{-1} \E\Big[\Big(\sum_{i=0}^{p-1} \wt Y_{i+1}\Big)^2 \Big]\,.
\end{equation}
If $j\neq i$ one can show by conditioning on $\mathcal{F}_{\max(i,j)}$ that 
$\E[\wt Y_{i+1} \wt Y_{j+1}]=0$. This in conjunction with the inequality $(a+b)^2\le 2(a^2+b^2)$ gives
\begin{equation*}
\begin{split}
\E\Big[\Big(\sum_{i=0}^{p-1} \wt Y_{i+1}\Big)^2 \Big]&= \sum_{i=0}^{p-1} \E[\wt Y_{i+1}^2]= \frac14 \sum_{i=0}^{p-1} \E\Big[(\wt Z_{i+1}^2-\E[\wt Z_{i+1}^2 | \mathcal{F}_i])^2\Big]\\
&\le \underbrace{\frac12 \sum_{i=0}^{p-1}  \E[\wt Z_{i+1}^4] }_{=o(1) \text{ by \eqref{eq:sumz4}}}
+ \frac12 \sum_{i=0}^{p-1} \E[(\E[\wt Z_{i+1}^2 | \mathcal{F}_i])^2] \Big)\\
&= o(1)+\frac12 \sum_{i=0}^{p-1} \E\Big[\Big\{ \frac{1-nS_2^{(i)}}{n-1} (1-n^2 \beta_4)+\underbrace{2 n^2 \beta_{2,2} \,\Big(\frac{1}{n-i}-S_2^{(i)}}_{\le c \, n^{-1} \text{ for some } c>0}\Big)
\Big\}^2\Big] \\
&\le o(1)+(1-n^2 \beta_4)^2 \sum_{i=0}^{p-1} \E\Big[\Big\{ \frac{1-nS_2^{(i)}}{n-1} 
\Big\}^2\Big]\\
&\le o(1)+ O\big(n^{3-\alpha +2 \vep}\big)\,,
\end{split}
\end{equation*}
where we used Lemma~\ref{lem:secondmoment} to obtain the third line, and Lemma~\ref{lem:moment24} in the last step. 

In view of \eqref{eq:boundY} and since $\alpha>3$, we have proved \eqref{sumY_i}.

\subsection{Proof of \eqref{sumR_i}}\label{sec:sumR_i}
We need the following lemma.
\begin{lemma}\label{lem:4.1}\cite[Lemma 4.1]{bao:pan:zhou:2015}
For $R_{i+1}$ defined in \eqref{eq:Ri+1} and $a>0$, if $\wt Z_{i+1}\ge -1+(\log n)^{-a}$ one has 
$$|R_{i+1}| \le C (U_{i+1}^2+|V_{i+1}|^{2+\delta}) \log \log n$$
for any $0\le \delta\le 1$. Here $C=C(a,\delta)$ is a positive constant that only depends on $a$ and $\delta$.
\end{lemma}
A combination of Lemmas \ref{lem:4.2} and \ref{lem:4.1} yields that, with probability $1-o(1)$, one has
\begin{equation}\label{eq:boundR}
\sum_{i=0}^{p-1}|R_{i+1}| \le C \sum_{i=0}^{p-1}(U_{i+1}^2+|V_{i+1}|^{2+\delta}) \log \log n\,, \qquad 0\le \delta\le 1\,.
\end{equation}
By \eqref{eq:boundR} and Markov's inequality, $\sum_{i=0}^{p-1} R_{i+1}\cip 0$  follows from
\begin{equation*}
\lim_{\nto} \log \log n \sum_{i=0}^{p-1}\Big(\E[U_{i+1}^2]+\E[|V_{i+1}|^{2+\delta}]\Big)=0\,, 
\end{equation*}
which in view of Lyapunov's inequality is implied by 
\begin{equation}\label{eq:boundR1}
\lim_{\nto} \log \log n \sum_{i=0}^{p-1}\Big(\E[U_{i+1}^2]+(\E[V_{i+1}^4])^{(2+\delta)/4}\Big)=0\,, 
\end{equation}
The $U$-part in \eqref{eq:boundR1} follows from Lemma~\ref{lem:asymptoticvariance}.

Finally  by Proposition~\ref{prop:offdiagbound}, we have, for any $\vep>0$ and $n$ sufficiently large, that $E[V_{i+1}^4] \le C\,n^{-\alpha/2+\vep}$, $0\le i\le p-1$, where the constant $C>0$ does not depend on $n$ and $i$. Therefore,
\begin{equation*}
\sum_{i=0}^{p-1} (\E[V_{i+1}^4])^{(2+\delta)/4} \le C^{(2+\delta)/4}\,p\,n^{(-\alpha/2+\vep)(2+\delta)/4}\,
\end{equation*}
With $\delta=1$ and using $p/n\to \gamma\in (0,1)$, the \rhs~is
$$C^{3/4}\frac{p}{n} n^{1-\frac{3\alpha}{8}+\frac{3\vep }{4}} \to 0\,,\qquad \nto\,,$$
for $\vep>0$ sufficiently small  since $\alpha>8/3$. 
This shows the $V$-part in \eqref{eq:boundR1} and completes the proof of \eqref{sumR_i}.

\subsection{Proof of \eqref{sumconst}}\label{sec:sumconst}

In view of \eqref{eq:reess}, equation \eqref{sumconst} follows from
\begin{equation}\label{eq:sufff}
 \sum_{i=0}^{p-1} \tfrac{1}{2}\E[\wt Z_{i+1}^2] -c_n+\mu_n=0\to 0\,,\qquad \nto\,.
\end{equation}
From Lemma~\ref{lem:asymptoticvariance}, we have 
\begin{equation*}
\sum_{i=0}^{p-1} \tfrac{1}{2}\E[\wt Z_{i+1}^2] = \sum_{i=0}^{p-1} \tfrac{\E[ U_{i+1}^2] }{2}+ \sum_{i=0}^{p-1} \tfrac{\E[ V_{i+1}^2] }{2}
\sim -\log(1-p/n)- p/n\,.
\end{equation*}
Recalling the definitions $\mu_n=(p-n+\frac{1}{2})\log(1-\frac{p}{n})-p +\frac{p}{n}$ and $c_n=-p \log n + \log(n(n-1)\cdots (n-p+1))$, \eqref{eq:sufff} is thus equivalent to
 \begin{equation}\label{eq:sufff1}
(p-n-\frac{1}{2})\log(1-p/n)- p - \sum\limits_{i=1}^{p-1}\log(1-i/n)\to 0\,,\qquad \nto\,.
\end{equation}

Taking the logarithm on both sides of the identity
	\begin{eqnarray*}
      \prod\limits_{i=1}^{p-1}\left(1-\frac{i}{n}\right)=\frac{n!(1-(p-1)/n)}{(n-p+1)!\, n^{p-1}}=\frac{n!(n-p+1)}{n(n-p+1)!\, n^{p-1}}=\frac{(n-1)!}{(n-p)!\, n^{p-1}}\,,
  \end{eqnarray*}
we get 
  \begin{eqnarray*}
    \sum\limits_{i=1}^p\log(1-i/n)=\log(n-1)!-(p-1)\log n-\log(n-p)! \,.
  \end{eqnarray*}
We approximate these terms using Stirling's formula $ \log(n!)=n\log n-n+\tfrac 12 \log(2\pi n)+O(n^{-1})$ and obtain
\begin{align*}
&(p-1)\log n - \log(n-1)! +\log(n-p)!= (p-1)\log n -(n-1)\log(n-1)+(n-1)\\
& \quad -\frac{\log(2\pi (n-1))}{2} 
+(n-p)\log(n-p)-(n-p)+\frac{\log(2\pi (n-p))}{2}+O(n^{-1})\\
&= p-1+(p-1)\log n -(n-\tfrac 12)\log(n-1)+(n-p+\tfrac 12)\log(n-p)+O(n^{-1})\\
&=p-1 +(n-\tfrac 12)\log(\tfrac{n}{n-1})+(n-p+\tfrac 12)\log(1-\tfrac pn)+O(n^{-1})\,.
\end{align*}	
Therefore, the \lhs~ in \eqref{eq:sufff1} is 
$-1 +(n-\tfrac 12)\log(\tfrac{n}{n-1})+O(n^{-1})$ which converges to zero as $\nto$.
	
This establishes \eqref{eq:sufff1} and thus finishes the proof of \eqref{sumconst}.

\subsection{Proof of Proposition~\ref{prop:important}}\label{sec:proofofprop:important}
First, we prove \eqref{eq:fourthV}.
Let $\alpha\in (2,4)$ and $\vep\in(0,\alpha/2-1)$. By Proposition~\ref{prop:offdiagbound} we have, for any $\delta>0$ and $n$ sufficiently large, that $E[V_{i+1}^4] \le C\,n^{-\alpha/2+\delta}$, $0\le i\le p-1$, where the constant $C>0$ does not depend on $n$. Therefore,
\begin{equation*}
n^{\vep}\sum_{i=0}^{p-1} \E[V_{i+1}^4] \le C\,p\,n^{-\alpha/2+\delta+\vep}\,
\end{equation*}
and using $p/n\to \gamma\in (0,1)$, the \rhs~converges to zero for sufficiently small $\delta>0$. This proves \eqref{eq:fourthV}.

Next, we turn to the proof of \eqref{eq:fourthU}. Let $\alpha\in (3,4)$ and $\vep\in(0,\alpha-3)$. From Corollary~\ref{cor:fourthmoment}, we know that for $0\le i\le p-1$,
\begin{align}
	&\E[U_{i+1}^4]=\E\E\Big[\Big(\sum_{k=1}^n q_{i,kk} (nY_{i+1,k}^2-1)\Big)^4 \, \Big| \mathcal{F}_i \Big] \nonumber \\
	&=\E[K_{4,4}^{(i)}]n^4\beta_{4,4}+\E[K_{2,2}^{(i)}]n^2\beta_{2,2}+\E[K_{2,2,2}^{(i)}]n^3\beta_{2,2,2}+\E[K_{2,2,2,2}^{(i)}]n^4\beta_{2,2,2,2}+ \E[K^{(i)}]\,, \label{eq:5.28}
  \end{align}
where $S_j^{(i)}=q_{i,11}^j+\cdots  +q_{i,nn}^j$, $j\ge 1$, and 
\begin{align}
K_{4,4}^{(i)}&=3(S_2^{(i)})^2-4S_3^{(i)}+nS_4^{(i)}\,,\qquad K^{(i)}=6nS_2^{(i)}-4n^2S_3^{(i)}+ n^3S_4^{(i)}-3\,, \notag\\
K_{2,2}^{(i)}&=-12 nS_2^{(i)}+8n^2 S_3^{(i)}-2n^3S_4^{(i)}+6\,, \label{k22}\\
K_{2,2,2}^{(i)}&=8nS_2^{(i)}-2n(S_2^{(i)})^2+\frac{8n(1-2n)}{3}S_3^{(i)}+\frac{2n^2(2n-1)}{3}S_4^{(i)}-4\,, \label{k222}\\
K_{2,2,2,2}^{(i)}&=-2nS_2^{(i)}+(2n-3)(S_2^{(i)})^2+\frac{4(n^2-2n+3)}{3}S_3^{(i)}-\frac{n(n^2-2n+3)}{3}S_4^{(i)}+1\,.\label{k2222}
\end{align}

By \eqref{eq:sumK}, we have
\begin{equation}\label{eq:sumKi}
K^{(i)}+K_{4,4}^{(i)}+K_{2,2}^{(i)}+K_{2,2,2}^{(i)}+K_{2,2,2,2}^{(i)}=0\,.
\end{equation}
Plugging \eqref{eq:sumKi} into \eqref{eq:5.28} gives
\begin{align}\label{eq:Z4}
	\E[U_{i+1}^4]
	&=\E[K_{44}^{(i)}](n^4\beta_{4,4}-1)+\E[K_{2,2}^{(i)}](n^2\beta_{2,2}-1) \notag\\ &\quad +\E[K_{2,2,2}^{(i)}](n^3\beta_{2,2,2}-1)+\E[K_{2,2,2,2}^{(i)}](n^4\beta_{2,2,2,2}-1).
  \end{align}
	We will bound the \rhs~ term by term.

Due to $p/n\to \gamma\in (0,1)$ it holds $(1-\gamma)n \sim n-p \le n-i\le n$, so that $n-i$ is of order $n$ for all $0\le i\le p-1$. A combination of this fact with \eqref{eq:boundelements} yields that for sufficiently large $n$ there exists a constant $c>1$ such that $|S_j^{(i)}|\le c^{1/2} n^{1-j}$. 
Thus we get
\begin{equation}\label{eq:k44}
|K_{4,4}^{(i)}|=|3(S_2^{(i)})^2-4S_3^{(i)}+nS_4^{(i)}|\le \frac{3c}{n^2}+\frac{4c}{n^2}+\frac{c}{n^2}=\frac{8c}{n^2}\,.
\end{equation}
Using \eqref{eq:k44}, for any $\vep>0$ and $n$ sufficiently large the first term is bounded by
\begin{equation*}
\Big| \sum_{i=0}^{p-1} \E[K_{44}^{(i)}](n^4\beta_{4,4}-1) \Big|\le \underbrace{|(n^4\beta_{4,4}-1)|}_{\le n^{4-\alpha+\vep}} \sum_{i=0}^{p-1} \E[|K_{44}^{(i)}|] =O(n^{3-\alpha+\vep})\,.
\end{equation*}
Note that $1- n^2\beta_{2,2},1- n^3\beta_{2,2,2},1- n^4\beta_{2,2,2,2}$ are nonnegative.
Thus,
\begin{equation*}
\begin{split}
\sum_{i=0}^{p-1} \E[U_{i+1}^4]&\le O(n^{3-\alpha+\vep})+
(1- n^2\beta_{2,2}) \Big| \sum_{i=0}^{p-1}\E[K_{2,2}^{(i)}] \Big|\\
& \quad+(1- n^3\beta_{2,2,2})\Big| \sum_{i=0}^{p-1}\E[K_{2,2,2}^{(i)}] \Big|+(1- n^4\beta_{2,2,2,2})\Big| \sum_{i=0}^{p-1}\E[K_{2,2,2,2}^{(i)}]\Big|\,.
\end{split}
\end{equation*}
Next, we turn to the remaining terms.   Since 
$$n^2\beta_{2,2}\sim n^3\beta_{2,2,2}\sim n^4\beta_{2,2,2,2}\sim 1 \,,\quad \nto\,,$$
it holds for any $\vep>0$,
\begin{equation*}
1- n^2\beta_{2,2}= 1- n(n-1)\beta_{2,2} + O(n^{-1})=n \beta_4 +O(n^{-1})=O(n^{1-\alpha/2+\vep})\,,
\end{equation*}
where also \eqref{id1} was used. Analogously, applying \eqref{id2}, \eqref{id3} and Lemma~\ref{lem:moment24}, we get for any $\vep>0$ that
\begin{equation*}
1- n^3\beta_{2,2,2}=O(n^{1-\alpha/2+\vep})\quad \text{ and } \quad 1- n^4\beta_{2,2,2,2}=O(n^{1-\alpha/2+\vep})\,.
\end{equation*}
Hence, \eqref{eq:fourthU} is proved if we can show that there exists an $\vep>0$ such that, as $\nto$,
\begin{equation*}
n^{1-\alpha/2+\vep}\Big| \sum_{i=0}^{p-1}\E[K_{2,2}^{(i)}] \Big|\to 0\,, \qquad n^{1-\alpha/2+\vep}\Big| \sum_{i=0}^{p-1}\E[K_{2,2,2}^{(i)}] \Big|\to 0\,, \qquad n^{1-\alpha/2+\vep}\Big| \sum_{i=0}^{p-1}\E[K_{2,2,2,2}^{(i)}] \Big|\to 0\,.
\end{equation*}
Fortunately, Lemma~\ref{lem:nestor3} verifies the latter. The proof is complete.


\appendix

\section{Offdiagonal part of a quadratic form}\label{sec:A}\setcounter{equation}{0}

\begin{proposition}\label{prop:offdiagbound}
Let $\alpha\in (2,4)$. Under the assumptions of Theorem~\ref{thm:main} we have for $\vep>0$ and $n$ sufficiently large,
\begin{equation}\label{eq:offdiagbound}
n^4 \E \Big[ \Big( \sum_{k\neq l} q_{i,kl} Y_{i+1,k} Y_{i+1,l}\Big)^4 \Big] \le C\,n^{-\alpha/2+\vep}\,, \qquad i=0,1,\ldots, p-1\,,
\end{equation}
where the constant $C>0$ does not depend on $n$ and $i$.
\end{proposition}

\begin{proof}
Let $0\le i\le p-1$ and $s= 4$. Throughout this proof, in the notation $\beta_{m_1,\ldots,m_r}$ we always assume $m_1+\cdots+m_r=s$.  
Using that $Y_{i+1,j}\eid -Y_{i+1,j}$, we have 
\begin{equation*}
\begin{split}
\E \Big[ \Big( \sum_{k_1\neq k_2} q_{i,k_1k_2} Y_{i+1,k_1} Y_{i+1,k_2}\Big)^s \Big| \mathcal{F}_i\Big]
&= \sum_{\substack{k_1\neq k_2,\ldots, k_{2s-1}\neq k_{2s}:\\ \sum_{t=1}^{2s}\delta_{k_j k_t} \text{ is even } \forall 1\le j\le 2s}} q_{i,k_1k_2}\cdots q_{i,k_{2s-1}k_{2s}} \, \E[Y_{i+1,k_1} \cdots Y_{i+1,k_{2s}}]\\
&= \sum_{r=2}^s \sum_{\bfk \in \calK_{r,s}} q_{i,\bfk} \beta_{\bfk} \\
&\le \sum_{r=2}^s \max_{\bfk \in \calK_{r,s}} \beta_{\bfk} \,\, \Big|\sum_{\bfk \in \calK_{r,s}} q_{i,\bfk} \Big|\,,
\end{split}
\end{equation*}
where $\beta_{\bfk}=\E[Y_{i+1,k_1} \cdots Y_{i+1,k_{2s}}]$ and $q_{i,\bfk}=q_{i,k_1k_2}\cdots q_{i,k_{2s-1}k_{2s}}$ for $\bfk = (k_1,\ldots,k_{2s})$, and 
\begin{equation*}
\calK_{r,s}= \Big\{(k_1,\ldots,k_{2s})\in \{1,\ldots,n\}^{2s} \,:\, 
\substack{\#\{k_1,\ldots,k_{2s}\}=r; ~ k_1\neq k_2,\ldots, k_{2s-1}\neq k_{2s};\\
 \sum_{t=1}^{2s}\delta_{k_j k_t} \text{ is even } \forall 1\le j\le 2s} \Big\}\,.
\end{equation*}
Here $\delta_{k_j k_t}$ is the Kronecker-delta, i.e., $\delta_{k_j k_t}=\1_{\{k_j= k_t\}}$. We will bound $\max_{\bfk \in \calK_{r,s}} \beta_{\bfk}$ and $|\sum_{\bfk \in \calK_{r,s}} q_{i,\bfk} |$.
\par

We start with the first term. By Lemma \ref{lem:moment24}, we have for integers $k_1,\ldots,k_r\ge 1$ that
\begin{equation}\label{eq:bound00}
\lim_{\nto} \frac{n^{N_1(1-\alpha/2)+ r\alpha/2}}{L^{r-N_1}(n^{1/2})}  \beta_{2k_1,\ldots, {2k_r}} \sim C(k_1,\ldots,k_r)  \,,
\end{equation}
where $N_1:=N_1(k_1,\ldots,k_r):=\#\{1\le j\le r: k_j=1\}$ and
$$C(k_1,\ldots,k_r):=\frac{(\alpha/2)^{r-N_1}\Gamma(N_1(1-\alpha/2)+ r\alpha/2) \, \prod_{i:k_i\ge 2} \Gamma(k_i-\alpha/2)}{\Gamma(k_1+\cdots+k_r)}.$$
Since $(\prod_j \Gamma(a_j))/\Gamma(\sum_j a_j) \le 1$ for $a_j\ge 0$ we observe that
\begin{equation}\label{eq:sdfasa}
C(k_1,\ldots,k_r)\le (\alpha/2)^{r-N_1}\le 2^{r-N_1}\,.
\end{equation}
We recall the Potter bounds on the regularly varying function $L\ge 0$. For any $\vep>0$ and sufficiently large $n$ it holds
\begin{equation}\label{eq:potter}
n^{-\vep}\le L((n^{1/2}))\le n^{\vep}\,.
\end{equation}
Choose $\vep \in (0,\alpha/2 -1)$. In view of \eqref{eq:bound00}-\eqref{eq:potter}, we have for sufficiently large $n$ that
\begin{equation}\label{eq:bound001}
\beta_{2k_1,\ldots, {2k_r}} \le n^{\alpha/2 (N_1- r)-N_1} L^{r-N_1}(n^{1/2}) 2^{r-N_1}\le   n^{-(\alpha/2-\vep) (r-N_1)-N_1} \,.
\end{equation}
Therefore, we obtain
\begin{equation*}
\begin{split}
\max_{\bfk \in \calK_{r,s}} \beta_{\bfk}&\le \max_{\substack{k_1,\ldots,k_r\ge 1: \\ k_1+\cdots+k_r=s}} \beta_{2k_1,\ldots, {2k_r}}\le n^{-r (\alpha/2-\vep)} \max_{\substack{k_1,\ldots,k_r\ge 1: \\ k_1+\cdots+k_r=s}} n^{N_1(\alpha/2-1-\vep) }\,.
\end{split}
\end{equation*}
Since $N_1\le r-\1_{\{r<s\}}$, we conclude that for large $n$,
\begin{equation}\label{eq:boundterm1}
\max_{\bfk \in \calK_{r,s}} \beta_{\bfk}\le n^{-r-(\alpha/2-1-\vep)\1_{\{r<s\}}}\,.
\end{equation}
This establishes a bound on $\max_{\bfk \in \calK_{r,s}}\beta_{\bfk}$. For later reference, we note that $\alpha/2-1-\vep>0$.
\par

Next, we turn to the bound of $|\sum_{\bfk \in \calK_{r,s}} q_{i,\bfk} |$.
Let $(X_j)_{j\ge 1}$ be an i.i.d. sequence (which is also independent of $\X$) with distribution $\P(X_j=1)=\P(X_j=-1)=1/2$. Using that $\E[X_j^t]=1$ if $t$ is even and zero otherwise, we have as above
\begin{equation}\label{eq:sdfsef}
\begin{split}
\E \Big[ \Big| \sum_{k_1\neq k_2} q_{i,k_1k_2} X_{k_1} X_{k_2}\Big|^s \Big| \mathcal{F}_i\Big]
&= \sum_{r=2}^s \sum_{\bfk \in \calK_{r,s}} q_{i,\bfk} \,.
\end{split}
\end{equation}
Applying Lemma \ref{lem:qmoment} with the sequence $(X_j)$, we get
\begin{equation*}
\E \Big[ \Big| \sum_{k_1\neq k_2} q_{i,k_1k_2} X_{k_1} X_{k_2}\Big|^s \Big| \mathcal{F}_i\Big]\le (Cs)^s\,  \Big(\sum_{k\neq l} q_{i,kl}^2\Big)^{s/2}
\end{equation*}
In view of \eqref{eq:sdfsef}, we see that
\begin{equation}\label{eq:plsegg}
\sum_{r=2}^s \sum_{\bfk \in \calK_{r,s}} q_{i,\bfk}\le \sum_{r=2}^s \sum_{\bfk \in \calK_{r,s}} |q_{i,\bfk}|\le (Cs)^s\,  \Big(\sum_{k\neq l} q_{i,kl}^2\Big)^{s/2}\,,
\end{equation}
where the last inequality follows from the fact that the \rhs~in \eqref{eq:erdos} remains the same if we replace $a_{ij}$ with $|a_{ij}|$. Here $C$ is an absolute constant that does not depend on $s$.

Since the \rhs~in \eqref{eq:boundterm1} depends on $r$, it is important find an upper bound on $|\sum_{\bfk \in \calK_{r,s}} q_{i,\bfk} |$ that uses the value of $r=2,\ldots, s$ as well. If $r=s$ we conclude from \eqref{eq:plsegg} and $\sum_{k,l} q_{i,kl}^2=(n-i)^{-1}$ that
\begin{equation}\label{eq:rs}
\Big |\sum_{\bfk \in \calK_{s,s}} q_{i,\bfk}\Big|\le (Cs)^s\,  \Big(\sum_{k\neq l} q_{i,kl}^2\Big)^{s/2}\le (Cs)^s\Big( \frac{1}{n-i}\Big)^{s/2}\,.
\end{equation}
Note that the term $(\sum_{k\neq l} q_{i,kl}^2)^{s/2}$ actually appears in $\sum_{\bfk \in \calK_{s,s}} q_{i,\bfk}$. Indeed, this follows directly from the definition of the latter sum by setting $k_1=k_3,k_2=k_4,\ldots, k_{2s-2}=k_{2s}$. Hence, the maximum number of distict indices $k_j$ in $q_{i,\bfk}$ and the maximum number of distinct indices in $(\sum_{k\neq l} q_{i,kl}^2)^{s/2}$ are both equal to $r$.
From the definition of $\calK_{r,s}$, recall that $\#\{k_1,\ldots,k_{2s}\}=r$ if $(k_1,\ldots,k_{2s})\in \calK_{r,s}$.

If $r=s-1$, we may thus restrict ourselves to $s-1$ distinct indices. Due to $q_{kl}=q_{lk}$, this yields the bound 
\begin{equation}\label{eq:rs-1}
\Big |\sum_{\bfk \in \calK_{s-1,s}} q_{i,\bfk}\Big|\le (Cs)^s  \Big(\sum_{k\neq l} q_{i,kl}^2\Big)^{s/2-2} \sum_{k_1\neq k_2} q_{i,k_1k_2}^2 \sum_{k_3=1; k_3\neq k_1}^n q_{i,k_1k_3}^2 \le (Cs)^s\Big( \frac{1}{n-i}\Big)^{s/2+1}\,,
\end{equation}
where the last inequality holds since $Q_i^2=Q_i /(n-i)$ and \eqref{eq:boundelements} imply
$$\sum_{l=1}^n q_{i,kl}^2=\frac{q_{i,kk}}{n-i}\le \Big( \frac{1}{n-i}\Big)^2\,.$$
From the definition of $\calK_{r,s}$ and \eqref{eq:boundelements} it follows for $r=2$ that
\begin{equation*}
\Big |\sum_{\bfk \in \calK_{2,s}} q_{i,\bfk}\Big|
= 2^s \Big |\sum_{k<l} q_{i,kl}^s\Big|\le \frac{2}{(n-i)^{s-2}} \sum_{k\neq l} q_{i,kl}^2 \le 2 \Big( \frac{1}{n-i}\Big)^{s-1}\,.
\end{equation*}
In combination with \eqref{eq:rs} and \eqref{eq:rs-1}, this yields that
\begin{equation}\label{eq:allr}
\Big |\sum_{\bfk \in \calK_{s-t,s}} q_{i,\bfk}\Big|\le (Cs)^s\Big( \frac{1}{n-i}\Big)^{s/2+ \lceil{t/2} \rceil}\,, \qquad t=0,\ldots,s-2 \,,
\end{equation}
where $\lceil{t/2} \rceil$ is the smallest integer greater or equal to $t/2$ and $C>0$ is a constant.

Finally, we complete the proof of the proposition. In view of \eqref{eq:boundterm1} and \eqref{eq:allr}, we get for $s= 4$ and sufficiently large $n$,
\begin{equation*}
\begin{split}
n^s \,\Big|\E \Big[ \Big( \sum_{k_1\neq k_2} q_{i,k_1k_2} Y_{i+1,k_1} Y_{i+1,k_2}\Big)^s \Big]\Big|&=
n^s\, \Big|\E \E \Big[ \Big( \sum_{k_1\neq k_2} q_{i,k_1k_2} Y_{i+1,k_1} Y_{i+1,k_2}\Big)^s \Big| \mathcal{F}_i\Big]\Big|\\
&\le n^s\, \E \sum_{r=2}^s \max_{\bfk \in \calK_{r,s}} \beta_{\bfk} \,\, \Big|\sum_{\bfk \in \calK_{r,s}} q_{i,\bfk} \Big|\\
&\le n^s\, \sum_{r=2}^s   n^{-r-(\alpha/2-1-\vep)\1_{\{r<s\}}}  \cdot (Cs)^s\Big( \frac{1}{n-i}\Big)^{s/2+ \lceil{(s-r)/2} \rceil}\\
&\le (\widetilde C s)^s \Big( n^{-s/2}+ \sum_{r=2}^{s-1}   n^{s-r-(\alpha/2-1-\vep)-s/2- \lceil{(s-r)/2} \rceil}  \Big)\\
&=  (\widetilde C s)^s \Big( n^{-s/2}+ n^{s/2-(\alpha/2-1-\vep)}\sum_{r=2}^{s-1}   n^{-r- \lceil{(s-r)/2} \rceil}  \Big)\\
&\le  (\widetilde C s)^s \Big( n^{-s/2}+ n^{s/2-(\alpha/2-1-\vep)}s \,  n^{-1-\lceil{s/2} \rceil}  \Big)\\
&\le \widetilde C^s s^{s+1}\,n^{-1-(\alpha/2-1-\vep)}
\end{split}
\end{equation*}
with some constant $\widetilde C>0$ that does not depend on $n$ or $s$. 
\end{proof}

\section{Additional technical lemmas}\label{sec:B}\setcounter{equation}{0}
The following lemmas are needed in the proof of our main result. Recall the matrix $Q_i=\{q_{i, kl}\}_{k,l=1}^n=P_i/(n-i)$, where the projection matrix $P_i =\bfI_n-B_{(i)}^{\top} (B_{(i)} B_{(i)}^\top)^{-1} B_{(i)}$ for $0\leq i\leq p-1$ and $P_0=\bfI_n$.

\begin{lemma}\label{lem:nestor3}
Let $\alpha\in (3,4)$. Under the conditions of Theorem \ref{thm:main}, there exists an $\vep>0$ such that, as $\nto$,
\begin{equation*}
n^{1-\alpha/2+\vep}\Big| \sum_{i=0}^{p-1}\E[K_{2,2}^{(i)}] \Big|\to 0\,, \qquad n^{1-\alpha/2+\vep}\Big| \sum_{i=0}^{p-1}\E[K_{2,2,2}^{(i)}] \Big|\to 0\,, \qquad n^{1-\alpha/2+\vep}\Big| \sum_{i=0}^{p-1}\E[K_{2,2,2,2}^{(i)}] \Big|\to 0\,,
\end{equation*}
where $S_j^{(i)}=q_{i,11}^j+\cdots  +q_{i,nn}^j$, $j\ge 1$, and $K_{2,2}^{(i)},K_{2,2,2}^{(i)},K_{2,2,2,2}^{(i)}$ are defined in \eqref{k22}, \eqref{k222} and \eqref{k2222}, respectively.
\end{lemma}
\begin{proof}
Let's rewrite $\E[K_{2,2}^{(i)}]$, $\E[K_{2,2,2}^{(i)}]$ and $\E[K_{2,2,2,2}^{(i)}] $ in the following way
\begin{align*}
  \E[K_{2,2}^{(i)}]&= -12n\sum\limits_{\ell=1}^n\left(\E[q^2_{i,\ell\ell}]-\frac{1}{n^2}\right)+8n^2\sum\limits_{\ell=1}^n\left(\E[q^3_{i,\ell\ell}]-\frac{1}{n^3}\right)-2n^3\sum\limits_{\ell=1}^n\left(\E[q^4_{i,\ell\ell}]-\frac{1}{n^4}\right),\\
\E[K_{2,2,2}^{(i)}]&= 8n\sum\limits_{\ell=1}^n\left(\E[q^2_{i,\ell\ell}]-\frac{1}{n^2}\right) -\frac{16n^2}{3}\sum\limits_{\ell=1}^n\left(\E[q^3_{i,\ell\ell}]-\frac{1}{n^3}\right)+\frac{4n^3}{3}\sum\limits_{\ell=1}^n\left(\E[q^4_{i,\ell\ell}]-\frac{1}{n^4}\right)+O(n^{-1}),\\
  \E[K_{2,2,2,2}^{(i)}]&= -2n\sum\limits_{\ell=1}^n\left(\E[q^2_{i,\ell\ell}]-\frac{1}{n^2}\right)+\frac{4n^2}{3}\sum\limits_{\ell=1}^n\left(\E[q^3_{i,\ell\ell}]-\frac{1}{n^3}\right)-\frac{n^3}{3}\sum\limits_{\ell=1}^n\left(\E[q^4_{i,\ell\ell}]-\frac{1}{n^4}\right) + O(n^{-1}) \,,
\end{align*}
where we have  used the fact that $|S_2^{(i)}|\leq C n^{-1}$ for some constant $C>1$. The application of Lemma \ref{lem:yaskov} for $k=2, 3, 4$  leads to
\begin{eqnarray*}
 n^{1-\alpha/2+\vep}\Big| \sum\limits_{i=0}^{p-1}\E[K_{2,2}^{(i)}]\Big| &\leq&  O\left(n^{1-\alpha/2+\vep}n^{1/2}\right)=O\left(n^{(3-\alpha)/2+\vep}\right)\,.
\end{eqnarray*}
Similarly, we get
\begin{eqnarray*}
   n^{1-\alpha/2+\vep}\Big| \sum\limits_{i=0}^{p-1}\E[K_{2,2,2}^{(i)}]\Big| &=&  O\left(n^{(3-\alpha)/2+\vep}\right),\\
   n^{1-\alpha/2+\vep}\Big| \sum\limits_{i=0}^{p-1}\E[K_{2,2,2,2}^{(i)}]\Big| &=&  O\left(n^{(3-\alpha)/2+\vep}\right)\,,
\end{eqnarray*}
which verifies the statement of the lemma by noting that $\alpha>3$.
\end{proof}

\begin{lemma}\label{lem:yaskov}
Under the conditions of Theorem \ref{thm:main}, it holds for all $k\geq 2$ that
\begin{equation}\label{eq:fdfdd}
  0\leq n^{k-2}\sum_{i=0}^{p-1} \sum_{\ell=1}^n \left(\E[q_{i,\ell \ell}^k] - \frac{1}{n^{k}}\right) \leq O(n^{-1/2})\,, \qquad \nto\,.
\end{equation}
\end{lemma}
\begin{proof}
  First, using Jensen's inequality and the fact that $\sum_{\ell=1}^n q_{i,\ell \ell}=1$ with $q_{i,ll}\geq 0$ we observe that
  \begin{eqnarray*}
   \frac{1}{n^k}=\left(\frac{1}{n}\sum_{\ell=1}^n q_{i,\ell \ell}\right)^k\leq \frac{1}{n}\sum_{\ell=1}^n q^k_{i,\ell \ell}\,,
  \end{eqnarray*}
  which implies that
  \begin{eqnarray}\label{eq:sdfsdd}
  \frac{1}{n^{k-1}}\leq \sum_{\ell=1}^n q^k_{i,\ell \ell}\,.
  \end{eqnarray}
  Then, using this lower bound it follows that
  \begin{eqnarray*}
    \frac{p}{n^{k-1}}\leq \sum_{i=0}^{p-1}\sum_{\ell=1}^n \left(q_{i,\ell \ell}-\frac{1}{n}+\frac{1}{n}\right)^k=
\frac{p}{n^{k-1}}+ \sum\limits_{i=0}^{p-1}\sum\limits_{\ell=1}^n\sum\limits_{j=0}^{k-1}\binom{k}{j}\left(q_{i,\ell \ell}-\frac{1}{n}\right)^{k-j}\frac{1}{n^{j}}
  \end{eqnarray*}
  and, thus, taking expectations yields
  \begin{eqnarray*}
   0&\leq& \sum_{i=0}^{p-1}\sum_{\ell=1}^n \E[q^k_{i,\ell \ell}]-\frac{p}{n^{k-1}}=
\sum\limits_{i=0}^{p-1}\sum\limits_{\ell=1}^n\sum\limits_{j=0}^{k-2}\binom{k}{j}\E\left(q_{i,\ell \ell}-\frac{1}{n}\right)^{k-j}\frac{1}{n^{j}}\\
&=&\sum\limits_{j=0}^{k-2}\sum\limits_{i=0}^{p-1}\sum\limits_{l=1}^n\binom{k}{j}\E\left(q_{i,\ell \ell}-\frac{1}{n}\right)^{k-j}\frac{1}{n^{j}} \,,
  \end{eqnarray*}
where we have used for $j=k-1$ the property $\sum_{\ell=1}^n \E\left(q_{i,\ell \ell}-\frac{1}{n}\right)=0$.
 Next we will show that for any $k\geq 2$
  \begin{eqnarray}\label{rate_k}
\sum_{i=0}^{p-1}\sum_{\ell=1}^n \E\left(q_{i,\ell \ell}-\frac{1}{n}\right)^k=O\left( n^{2-k-1/2} \right)\,,    
  \end{eqnarray}
which will in fact imply that every term $\sum\limits_{i=0}^{p-1}\sum\limits_{\ell=1}^n\binom{k}{j}\E\left(q_{i,\ell \ell}-\frac{1}{n}\right)^{k-j}\frac{1}{n^{j}}$ will have the same order as the first one, i.e., for $j=0$, and, thus, because $k$ is fixed, we will get
\begin{eqnarray*}
n^{k-2} \sum\limits_{j=0}^{k-2} \sum\limits_{i=0}^{p-1}\sum\limits_{\ell=1}^n\binom{k}{j}\E\left(q_{i,\ell \ell}-\frac{1}{n}\right)^{k-j}\frac{1}{n^{j}}=O\left( n^{-1/2 } \right)\,.
\end{eqnarray*}

We define for any $k\geq 2$
  \begin{eqnarray*}
\delta_n:=\delta_{n,k}:=\sum_{i=0}^{p-1}\sum_{\ell=1}^n \E\left(q_{i,\ell \ell}-\E(q_{i,\ell \ell})\right)^k\,,
  \end{eqnarray*}
  where $\E(q_{i,\ell \ell})=\frac{1}{n}$, which follows from the fact that $q_{i,\ell \ell}$ are identically distributed over $\ell$ and $q_{i,11}+\cdots+q_{i,nn}=1$.
  Hence, it is enough to show that $\delta_n=O\big( n^{2-k-1/2} \big)$.
 Denote for $\ell=1,\ldots,n$ the vector $v_{\ell,i}$ as the $\ell$-th column of the matrix $B_{(i)}$.  First, we note that for all $\ell=1,\ldots,n$ it holds
  \begin{eqnarray*}
   p_{i,\ell\ell}=1-v^\top_{\ell,i}(B_{(i)} B_{(i)}^\top)^{-1}v_{\ell,i}\,.
  \end{eqnarray*}
 Denote now $\tilde{p}_{i,\ell\ell}=1-p_{i,\ell\ell}$ and use Minkowski's inequality to get
  \begin{eqnarray*}
    \delta_n&=&\sum_{i=0}^{p-1}\frac{n}{(n-i)^k} \E\left(p_{i,11}-\E(p_{i,11})\right)^k\le\sum_{i=0}^{p-1}\frac{n}{(n-i)^k} \E\left|\tilde{p}_{i,11}-\E(\tilde{p}_{i,11})\right|^k\\
 &=& \sum_{i=0}^{p-1}\frac{n}{(n-i)^k} \E\left|{v}^\top_{1,i}({B}_{(i)}{B}_{(i)}^\top)^{-1}{v}_{1,i}-\E({v}^\top_{1,i}({B}_{(i)} {B}_{(i)}^\top)^{-1}{v}_{1,i})\right|^k  \overset{\text{Minkowski}}{\leq} C\left(\delta^{(1)}_n+\delta^{(2)}_n+\delta^{(3)}_n\right)
  \end{eqnarray*}
  with some constant $C>0$ possibly depending on $k$, whose value is not important and may change from line to line, and 
  \begin{eqnarray*}
    \delta^{(1)}_n&=&\sum_{i=0}^{p-1}\frac{n}{(n-i)^k} \E\left|{v}^\top_{1,i}({B}_{(i)}{B}_{(i)}^\top)^{-1}{v}_{1,i}-{v}^\top_{1,i}({B}_{(i)}{B}_{(i)}^\top+\epsilon_nn\bfI_i)^{-1}{v}_{1,i}\right|^k,\\
    \delta^{(2)}_n&=& \sum_{i=0}^{p-1}\frac{n}{(n-i)^k} \E\left|{v}^\top_{1,i}({B}_{(i)}{B}_{(i)}^\top+\epsilon_nn\bfI_i)^{-1}{v}_{1,i}-\frac{\E\tr({B}_{(i,1)}{B}_{(i,1)}^\top+\epsilon_nn\bfI_i)^{-1}}{1+\E\tr({B}_{(i,1)}{B}_{(i,1)}^\top+\epsilon_nn\bfI_i)^{-1}}\right|^k,\\
    \delta^{(3)}_n&=&\sum_{i=0}^{p-1}\frac{n}{(n-i)^k} \left|\frac{\E\tr({B}_{(i,1)}{B}_{(i,1)}^\top+\epsilon_nn\bfI_i)^{-1}}{1+\E\tr({B}_{(i,1)}{B}_{(i,1)}^\top+\epsilon_nn\bfI_i)^{-1}}-\E({v}^\top_{1,i}({B}_{(i)}{B}_{(i)}^\top)^{-1}{v}_{1,i})\right|^k,
  \end{eqnarray*}
  where $\epsilon_n$ is a sequence tending to zero arbitrarily slower than $1/n$ and $B_{(i,1)}$ denotes the matrix obtained from $B_{(i)}$ by deleting the $1$st column $v_{1,i}$.  Let's consider $\delta_n^{(1)}$ first. It holds
  \begin{align}
     \Big|{v}^\top_{1,i}&({B}_{(i)}{B}_{(i)}^\top)^{-1}{v}_{1,i}-{v}^\top_{1}({B}_{(i)}{B}_{(i)}^\top+\epsilon_nn\bfI_i)^{-1}{v}_{1,i}\Big|= \epsilon_nn\cdot{v}^\top_{1,i}({B}_{(i)}{B}_{(i)}^\top)^{-1}({B}_{(i)}{B}_{(i)}^\top+\epsilon_nn\bfI_i)^{-1}{v}_{1,i} \nonumber\\
&\leq \frac{\epsilon_nn}{\lambda_{min}({B}_{(i)}{B}_{(i)}^\top+\epsilon_nn\bfI_i)}\underbrace{{v}^\top_{1,i}({B}_{(i)}{B}_{(i)}^\top)^{-1}{v}_{1,i}}_{=\tilde{p}_{i,11}\leq 1}\nonumber\\
    & \leq \frac{\epsilon_nn}{\lambda_{min}({B}_{(i)}{B}_{(i)}^\top)}\sim \frac{\epsilon_n}{(1-\sqrt{\frac{i}{n}})^2}\leq \frac{\epsilon_n}{(1-\sqrt{\frac{p}{n}})^2}\le C \epsilon_n\label{delta1} \,.
  \end{align}
  Thus, for $\delta_n^{(1)}$ and sufficiently large $n$, we have 
  \begin{eqnarray*}
    \delta_n^{(1)}\le C^k \sum_{i=0}^{p-1}\frac{n}{(n-i)^k} \epsilon^k_n \leq \underbrace{\frac{np}{(n-p+1)^2}}_{=O(1)} O(\epsilon^k_nn^{-k+2})=O(\epsilon^k_nn^{-(k-2)})\,.
  \end{eqnarray*}
  Now we proceed to $\delta_n^{(2)}$. Let's consider the following expression
  \begin{eqnarray}
   && \E\left|{v}^\top_{1,i}({B}_{(i)}{B}_{(i)}^\top+\epsilon_nn\bfI_i)^{-1}v_{1,i}-\frac{\E\tr({B}_{(i,1)}{B}_{(i,1)}^\top+\epsilon_nn\bfI_i)^{-1}}{1+\E\tr({B}_{(i,1)}{B}_{(i,1)}^\top+\epsilon_nn\bfI_i)^{-1}}\right|^k\nonumber\\
    &=&    \E\left|\frac{{v}^\top_{1,i}({B}_{(i,1)}{B}_{(i,1)}^\top+\epsilon_nn\bfI_i)^{-1}v_{1,i}}{1+{v}^\top_{1,i}({B}_{(i,1)}{B}_{(i,1)}^\top+\epsilon_nn\bfI_i)^{-1}v_{1,i}}-\frac{\E\tr({B}_{(i,1)}{B}_{(i,1)}^\top+\epsilon_nn\bfI_i)^{-1}}{1+\E\tr({B}_{(i,1)}{B}_{(i,1)}^\top+\epsilon_nn\bfI_i)^{-1}}\right|^k\nonumber\\
    &=&  \E\frac{\left|{v}^\top_{1,i}({B}_{(i,1)}{B}_{(i,1)}^\top+\epsilon_nn\bfI_i)^{-1}v_{1,i}-\E\tr({B}_{(i,1)}{B}_{(i,1)}^\top+\epsilon_nn\bfI_i)^{-1}\right|^k}{(1+\E\tr({B}_{(i,1)}{B}_{(i,1)}^\top+\epsilon_nn\bfI_i)^{-1})^k(1+{v}^\top_{1,i}({B}_{(i,1)}{B}_{(i,1)}^\top+\epsilon_nn\bfI_i)^{-1}v_{1,i})^k}\nonumber\\
    &\leq&  \E\left|{v}^\top_{1,i}({B}_{(i,1)}{B}_{(i,1)}^\top+\epsilon_nn\bfI_i)^{-1}v_{1,i}-\E\tr({B}_{(i,1)}{B}_{(i,1)}^\top+\epsilon_nn\bfI_i)^{-1}\right|^k\nonumber\\  
&\leq& C\left(\E\left|{v}^\top_{1,i}({B}_{(i,1)}{B}_{(i,1)}^\top+\epsilon_nn\bfI_i)^{-1}v_{1,i} - \tr({B}_{(i,1)}{B}_{(i,1)}^\top+\epsilon_nn\bfI_i)^{-1}\right|^k\right.\label{delta21}\\
& +& \left.\E\left| \tr({B}_{(i,1)}{B}_{(i,1)}^\top+\epsilon_nn\bfI_i)^{-1} - \E\tr({B}_{(i,1)}{B}_{(i,1)}^\top+\epsilon_nn\bfI_i)^{-1}\right|^k \right).\label{delta22}
  \end{eqnarray}
First, we analyze the term in \eqref{delta22} and 
define $\E_\ell:=\E(\cdot|v_{\ell,i},\ldots, v_{n,i})$ for $\ell=1,\ldots,n$ and $\E_{n+1}:=\E$. It holds

\vspace{-0.5cm}

  \begin{eqnarray}
&&    \tr({B}_{(i,1)}{B}_{(i,1)}^\top+\epsilon_nn\bfI_i)^{-1}-\E\tr({B}_{(i,1)}{B}_{(i,1)}^\top+\epsilon_nn\bfI_i)^{-1}\nonumber\\
    &=&\sum_{\ell=2}^n(\E_\ell-\E_{\ell+1})\tr({B}_{(i,1)}{B}_{(i,1)}^\top+\epsilon_nn\bfI_i)^{-1}\nonumber\\
    &=& \sum_{\ell=2}^n(\E_\ell-\E_{\ell+1})\left(\tr({B}_{(i,1)}{B}_{(i,1)}^\top+\epsilon_nn\bfI_i)^{-1} - \tr({B}_{(i,1)}{B}_{(i,1)}^\top-{v}_{\ell,i}{v}^\top_{\ell,i}+\epsilon_nn\bfI_i)^{-1} \right)\label{martingale_diff}\,,
  \end{eqnarray}
where the properties $\E_\ell(\tr({B}_{(i,1)}{B}_{(i,1)}^\top-{v}_{\ell,i}{v}^\top_{\ell,i}+\epsilon_nn\bfI_i)^{-1})=\E_{\ell+1}(\tr({B}_{(i,1)}{B}_{(i,1)}^\top-{v}_{\ell,i}{v}^\top_{\ell,i}+\epsilon_nn\bfI_i)^{-1})$ and  $\E_2(\tr({B}_{(i,1)}{B}_{(i,1)}^\top+\epsilon_nn\bfI_i)^{-1})=\tr({B}_{(i,1)}{B}_{(i,1)}^\top+\epsilon_nn\bfI_i)^{-1}$   were used.  Together with the definition of the  martingale differences sequence and Sherman-Morrison formula it implies
  \begin{eqnarray}
  && \E\left|\tr({B}_{(i,1)}{B}_{(i,1)}^\top+\epsilon_nn\bfI_i)^{-1}-\E\tr({B}_{(i,1)}{B}_{(i,1)}^\top+\epsilon_nn\bfI_i)^{-1}\right|^k\nonumber\\
    &\leq& C\sum_{\ell=2}^n\E\left|(\E_\ell-\E_{\ell+1})\left(\tr({B}_{(i,1)}{B}_{(i,1)}^\top+\epsilon_nn\bfI_i)^{-1} - \tr({B}_{(i,1)}{B}_{(i,1)}^\top-{v}_{\ell,i}{v}^\top_{\ell,i}+\epsilon_nn\bfI_i)^{-1}\right)\right|^k\nonumber\\
    &=&C\sum_{\ell=2}^n\E\left|(\E_\ell-\E_{\ell+1})\frac{{v}^\top_{\ell,i}({B}_{(i,1)}{B}_{(i,1)}^\top-{v}_{\ell,i}{v}^\top_{\ell,i}+\epsilon_nn\bfI_i)^{-2}{v}_{\ell,i}}{1+{v}^\top_{\ell,i}({B}_{(i,1)}{B}_{(i,1)}^\top-{v}_{\ell,i} {v}^\top_{\ell,i}+\epsilon_nn\bfI_i)^{-1}{v}_{\ell,i}}\right|^k\nonumber\\
&\leq&   \frac{C}{\epsilon^k_nn^{k}}\sum_{\ell=2}^n\E\Bigg|(\E_\ell-\E_{\ell+1})\underbrace{\frac{{v}^\top_{\ell,i}({B}_{(i,1)}{B}_{(i,1)}^\top-{v}_{\ell,i}{v}^\top_{\ell,i}+\epsilon_nn\bfI_i)^{-1}{v}_{\ell,i}}{1+{v}^\top_{\ell,i}({B}_{(i,1)}{B}_{(i,1)}^\top-{v}_{\ell,i} {v}^\top_{\ell,i}+\epsilon_nn\bfI_i)^{-1}{v}_{\ell,i}}}_{\leq 1}\Bigg|^k\nonumber\\
    &\leq&\frac{C(n-1) 2^k}{\epsilon^k_nn^k}=O(\epsilon^{-k}_nn^{-(k-1)})\label{delta2}\,.
  \end{eqnarray}
Next, we turn to \eqref{delta21}. 
To this end, we show that ${v}^\top_{1,i}({B}_{(i,1)}{B}_{(i,1)}^\top+\epsilon_nn\bfI_i)^{-1}v_{1,i}$ is bounded.
 Using the Sherman-Morrison formula we get

  \begin{eqnarray*}
 &&   {v}^\top_{1,i}({B}_{(i,1)}{B}_{(i,1)}^\top+\epsilon_nn\bfI_i)^{-1}v_{1,i}=  {v}^\top_{1,i}({B}_{(i)}{B}_{(i)}^\top+\epsilon_nn\bfI_i)^{-1}v_{1,i}+\frac{({v}^\top_{1,i}({B}_{(i)}{B}_{(i)}^\top+\epsilon_nn\bfI_i)^{-1}v_{1,i})^2}{1-{v}^\top_{1,i}({B}_{(i)}{B}_{(i)}^\top+\epsilon_nn\bfI_i)^{-1}v_{1,i}} \\
&=& \frac{{v}^\top_{1,i}({B}_{(i)}{B}_{(i)}^\top+\epsilon_nn\bfI_i)^{-1}v_{1,i}}{1-{v}^\top_{1,i}({B}_{(i)}{B}_{(i)}^\top+\epsilon_nn\bfI_i)^{-1}v_{1,i}} \leq \frac{1}{1-\kappa}
  \end{eqnarray*}
because ${v}^\top_{1,i}({B}_{(i)}{B}_{(i)}^\top+\epsilon_nn\bfI_i)^{-1}v_{1,i}\leq \kappa$ for some $\kappa<1$, for which Lemma A.1 in the Appendix of \cite{anatolyev:yaskov:2017} was used.  
Thus  $0\leq {v}^\top_{1,i}({B}_{(i,1)}{B}_{(i,1)}^\top+\epsilon_nn\bfI_i)^{-1}v_{1,i}\leq  C$, so it is enough to consider the case $k=2$. Indeed, let $k\geq 3$, then we get
\begin{eqnarray}
&& \E\left| {v}^\top_{1,i}({B}_{(i,1)}{B}_{(i,1)}^\top+\epsilon_nn\bfI_i)^{-1}v_{1,i} - \tr({B}_{(i,1)}{B}_{(i,1)}^\top+\epsilon_nn\bfI_i)^{-1}\right|^k\nonumber\\
&=& \E\left| {v}^\top_{1,i}({B}_{(i,1)}{B}_{(i,1)}^\top+\epsilon_nn\bfI_i)^{-1}v_{1,i} - \E_2{v}^\top_{1,i}({B}_{(i,1)}{B}_{(i,1)}^\top+\epsilon_nn\bfI_i)^{-1}v_{1,i}\right|^2\nonumber\\
&\times&\left| {v}^\top_{1,i}({B}_{(i,1)}{B}_{(i,1)}^\top+\epsilon_nn\bfI_i)^{-1}v_{1,i} - \E_2{v}^\top_{1,i}({B}_{(i,1)}{B}_{(i,1)}^\top+\epsilon_nn\bfI_i)^{-1}v_{1,i}\right|^{k-2}\nonumber\\
&\leq&
C\cdot \E\left| {v}^\top_{1,i}({B}_{(i,1)}{B}_{(i,1)}^\top+\epsilon_nn\bfI_i)^{-1}v_{1,i} - \E_2{v}^\top_{1,i}({B}_{(i,1)}{B}_{(i,1)}^\top+\epsilon_nn\bfI_i)^{-1}v_{1,i}\right|^2\label{delta_21}  \,. 
\end{eqnarray} 
Further we truncate the elements of the vector $v_{1,i}$ at the level $\zeta_n \sqrt{n}$, i.e., denote $\hat v_{1j,i}=v_{1j,i}\cdot\1\{|v_{1j,i}|\leq \zeta_n\sqrt{n}\}$ with $\zeta_n$ arbitrarily slow converging to zero but no faster than $n^{-1/2}$, i.e., $\zeta_n\sqrt{n}\to+\infty$. Note that  $\hat v_{1j,i}$ has mean zero since $v_{1j,i}$ is symmetrically distributed. Because the third absolute moment of $v_{1j,i}$ is finite it holds
\begin{eqnarray*}
  \lim\limits_{n\to\infty} \frac{\E|v_{1j,i}|^3\1\{|v_{1j,i}|>\zeta_n\sqrt{n}\}}{\zeta^3_n}=0\,.
\end{eqnarray*}
Moreover, one can also check that the second moment of $\hat{v}_{1j,i}$ converges to 1, indeed let $\Var(\hat{v}_{11,i})=\sigma^2_n$ then
\begin{eqnarray}\label{sigma2}
  |\sigma_n^2-1| \leq  C\E(|v_{11,i}|^2 \1\{|v_{11,i}|>\zeta_n\sqrt{n}\})\leq C \zeta^2_nn\frac{\E(|v_{11,i}|^3\1\{|v_{11,i}|>\zeta_n\sqrt{n}\})}{\zeta^3_nn^{3/2}}=o\left(n^{-1/2}\right)\,. 
\end{eqnarray}
Then the difference between the truncated and original quadratic forms is given by
\begin{eqnarray}\label{truncated_vs_original}
&& \E\left| {v}^\top_{1,i}({B}_{(i,1)}{B}_{(i,1)}^\top+\epsilon_nn\bfI_i)^{-1}v_{1,i} - \hat{v}^\top_{1,i}({B}_{(i,1)}{B}_{(i,1)}^\top+\epsilon_nn\bfI_i)^{-1}\hat{v}_{1,i}\right|^2\nonumber\\
&=&   \E\left| ({v}_{1,i}-\hat{v}_{1,i})^\top({B}_{(i,1)}{B}_{(i,1)}^\top+\epsilon_nn\bfI_i)^{-1}({v}_{1,i}+\hat{v}_{1,i})\right|^2\nonumber\\
&\overset{\text{CS}}{\leq} & C\E ({v}_{1,i}-\hat{v}_{1,i})^\top({B}_{(i,1)}{B}_{(i,1)}^\top+\epsilon_nn\bfI_i)^{-1}({v}_{1,i}-\hat{v}_{1,i})\cdot({v}_{1,i}+\hat{v}_{1,i})^\top({B}_{(i,1)}{B}_{(i,1)}^\top+\epsilon_nn\bfI_i)^{-1}({v}_{1,i}+\hat{v}_{1,i})\nonumber\\
&\leq& C \E({v}_{1,i}-\hat{v}_{1,i})^\top({B}_{(i,1)}{B}_{(i,1)}^\top+\epsilon_nn\bfI_i)^{-1}({v}_{1,i}-\hat{v}_{1,i})\leq C\E \frac{({v}_{1,i}-\hat{v}_{1,i})^\top({v}_{1,i}-\hat{v}_{1,i})}{\lambda_{min}({B}_{(i,1)}{B}_{(i,1)}^\top+\epsilon_nn\bfI_i)}\nonumber \\
&\leq &  \frac{C}{n} \E\left(\sum\limits_{j=1}^nv^2_{1j,i}\1 \{|v_{1j,i}|>\zeta_n\sqrt{n}\}\right) =C\E\left(|v_{11,i}|^2 \1\{|v_{11,i}|>\zeta_n\sqrt{n}\}\right)\nonumber\\
&\leq& C\zeta_n^2n\frac{\E\left(|v_{11,i}|^3\1\{|v_{11,i}|>\zeta_n\sqrt{n}\}\right)}{\zeta^3_nn^{3/2}}=o(n^{-1/2})\,.
\end{eqnarray}

Thus, we can safely replace $v_{1,i}$ by $\hat{v}_{1,i}$ in \eqref{delta_21}. 
We recall that $X_{11}$ is regularly varying with index $\alpha\in (3,4)$, i.e., $\P(|X_{11}>x)=x^{-\alpha} L(x)$ for a slowly varying function $L$. The following formula for truncated moments of $X_{11}$ is well--known (see, for instance, \cite{bingham:goldie:teugels:1987})
$$\hat{\nu}_4:=\E[|X_{11}|^4\1\{|X_{11}|>\zeta_n\sqrt{n}\}] 
\sim \frac{\alpha}{4-\alpha}\zeta^{4-\alpha}_nn^{(4-\alpha)/2} L(\zeta_n\sqrt{n})\,, \quad \nto\,.$$ 
Consider now \eqref{delta_21}
\begin{eqnarray}
&&  \E\left| v^\top_{1,i}({B}_{(i,1)}{B}_{(i,1)}^\top+\epsilon_nn\bfI_i)^{-1}v_{1,i} - \E_2 {v}^\top_{1,i}({B}_{(i,1)}{B}_{(i,1)}^\top+\epsilon_nn\bfI_i)^{-1} v_{1,i}\right|^2\nonumber\\
&\leq& C\underbrace{\E\left| v^\top_{1,i}({B}_{(i,1)}{B}_{(i,1)}^\top+\epsilon_nn\bfI_i)^{-1}v_{1,i} - \hat{v}^\top_{1,i}({B}_{(i,1)}{B}_{(i,1)}^\top+\epsilon_nn\bfI_i)^{-1} \hat{v}_{1,i}\right|^2}_{o(n^{-1/2})~\text{by \eqref{truncated_vs_original}}}\nonumber\\
& +& C\underbrace{\E\left|\E_2 v^\top_{1,i}({B}_{(i,1)}{B}_{(i,1)}^\top+\epsilon_nn\bfI_i)^{-1}v_{1,i} - \E_2\hat{v}^\top_{1,i}({B}_{(i,1)}{B}_{(i,1)}^\top+\epsilon_nn\bfI_i)^{-1} \hat{v}_{1,i}\right|^2}_{o(n^{-1})~\text{by \eqref{sigma2}}}  \nonumber\\
&+& C\E\left| \hat{v}^\top_{1,i}({B}_{(i,1)}{B}_{(i,1)}^\top+\epsilon_nn\bfI_i)^{-1}\hat{v}_{1,i} - \E_2\hat{v}^\top_{1,i}({B}_{(i,1)}{B}_{(i,1)}^\top+\epsilon_nn\bfI_i)^{-1} \hat{v}_{1,i}\right|^2 \label{last_summand}
\end{eqnarray}
 and apply Lemma B.26 from \cite{bai:silverstein:2010} on the last summand in \eqref{last_summand}
\begin{eqnarray*}
&&  \E\left| \hat{v}^\top_{1,i}({B}_{(i,1)}{B}_{(i,1)}^\top+\epsilon_nn\bfI_i)^{-1}\hat{v}_{1,i} - \E_2 \hat {v}^\top_{1,i}({B}_{(i,1)}{B}_{(i,1)}^\top+\epsilon_nn\bfI_i)^{-1} \hat v_{1,i}\right|^2\\
&\leq& C\hat{\nu}_4\E\tr({B}_{(i,1)}{B}_{(i,1)}^\top+\epsilon_nn\bfI_i)^{-2}\le \frac{C}{n} \hat{\nu}_4 \\
&\sim& \frac{C}{n} \zeta^{4-\alpha}_nn^{(4-\alpha)/2} L(\zeta_n\sqrt{n})=o(n^{-1/2})\,, \qquad \nto\,,
\end{eqnarray*}
where in the last step we used the Potter bounds for the slowly varying function $L$.

  Thus, similarly as for $\delta^{(1)}_n$  we get
  \begin{eqnarray}\label{delta_2}
    n^{k-2}\delta^{(2)}_n=O(\epsilon^{-k}_nn^{-(k-1)})+o(n^{-1/2})\,.
  \end{eqnarray}
  Concerning $\delta^{(3)}_n$, we observe the following
  \begin{eqnarray*}
 &&   \left|\frac{\E\tr({B}_{(i,1)}{B}_{(i,1)}^\top+\epsilon_nn\bfI_i)^{-1}}{1+\E\tr({B}_{(i,1)}{B}_{(i,1)}^\top+\epsilon_nn\bfI_i)^{-1}}-\E({v}^\top_{1,i}({B}_{(i)}{B}_{(i)}^\top)^{-1}v_{1,i})\right|^k \\
    &\leq& C \left|\E({v}^\top_{1,i}({B}_{(i)}{B}_{(i)}^\top+\epsilon_nn\bfI_i)^{-1}v_{1,i}-\frac{\E\tr({B}_{(i,1)}{B}_{(i,1)}^\top+\epsilon_nn\bfI_i)^{-1}}{1+\E\tr({B}_{(i,1)}{B}_{(i,1)}^\top+\epsilon_nn\bfI_i)^{-1}})\right|^k\\[0.2cm]
    &+&C \underbrace{|\E({v}^\top_{1,i}({B}_{(i)}{B}_{(i)}^\top)^{-1}v_{1,i}-\E({v}^\top_{1,i}({B}_{(i)}{B}_{(i)}^\top+\epsilon_nn\bfI_i)^{-1}v_{1,i}|^k}_{\text{$=O(\epsilon_n^k)$ due to \eqref{delta1}}}\\
        &\overset{Jensen}{\leq}& C\underbrace{\E\left| ({v}^\top_{1,i}({B}_{(i)}{B}_{(i)}^\top+\epsilon_nn\bfI_i)^{-1}{v}_{1,i}-\frac{\E\tr({B}_{(i,1)}{B}_{(i,1)}^\top+\epsilon_nn\bfI_i)^{-1}}{1+\E\tr({B}_{(i,1)}{B}_{(i,1)}^\top+\epsilon_nn\bfI_i)^{-1}})\right|^k}_{=O(\epsilon^{-k}_n n^{-(k-1)})+o(n^{-1/2})~\text{due to \eqref{delta_2}}}+O(\epsilon_n^k)
  \end{eqnarray*}
 and, as a result, we have
  \begin{eqnarray}
   n^{k-2} \delta^{(3)}_n=O(\epsilon_n^k)+O(\epsilon^{-k}_n n^{-(k-1)})+o(n^{-1/2})\,
  \end{eqnarray}
  and altogether we receive the following rate for $\delta_n$
  \begin{eqnarray}\label{delta}
    n^{k-2} \delta_n=O(\epsilon_n^k)+O(\epsilon^{-k}_n n^{-(k-1)})+o(n^{-1/2})\,.
  \end{eqnarray}
  Now we need to specify the sequence $\epsilon_n$ such that $\delta_n$ converges to zero as fast as possible. Because $\epsilon_n$ can not vanish faster than $1/n$ we assume w.l.o.g. that $\epsilon_n=n^{-\varepsilon}$ for some $0<\varepsilon<1$, plug it into \eqref{delta} and get
  \begin{eqnarray}
    n^{k-2}\delta_n=O(n^{-k\varepsilon})+O( n^{k(\varepsilon-1)+1})+o(n^{-1/2})\,.
  \end{eqnarray}
Choosing $\varepsilon =\frac{1}{2k}$ finishes the proof of the lemma.
\end{proof}


\begin{lemma}\cite[Lemma 7.10]{erdos:yau:2017}\label{lem:qmoment}
Let $X_1,\ldots,X_N$ be independent centered random variables and assume that 
\begin{equation*}
(\E[|X_i|^s])^{1/s} \le \mu_s\,,\quad 1\le i\le N; s=2,3,\ldots
\end{equation*} 
for some fixed constants $\mu_s$. Then we have for any deterministic complex numbers $a_{ij}, 1\le i,j\le N$ that 
\begin{equation}\label{eq:erdos}
\Big(\E\Big[ \Big|\sum_{i\neq j=1}^N a_{ij}X_i X_j \Big|^s \Big]\Big)^{1/s}\le C \, s\, \mu_s^2 \Big(\sum_{i\neq j=1}^N |a_{ij}|^2\Big)^{1/2}\,,\quad s=2,3,\ldots,
\end{equation} 
where the constant $C$ does not depend on $s$.
\end{lemma}

\begin{lemma} \cite[Theorem b) and d)]{wiens:1992} \label{lem:quf}
  Let $\bfz=(Z_1,\ldots,Z_n)^\top$ be a random vector such that, for all nonnegative integers $m_1,\ldots, m_6$ with $m_1+\cdots+m_6\le 6$, $\E[Z_{i_1}^{m_1}Z_{i_2}^{m_2} \cdots Z_{i_6}^{m_6}]$ is 
(i) finite;
(ii) zero if any $m_i$ is odd; and
(iii) invariant under permutations of the indices.
Let $\beta_2=\E[Z_1^2], \beta_{2,2}=\E[Z_1^2Z_2^2],\beta_4=\E[Z_1^4], \beta_{4,2}=\E[Z_1^4Z_2^2]$ and $\beta_6=\E[Z_1^6]$. 	
 Then we have for any real-valued and symmetric $n\times n$ nonrandom matrix $\bfA$ that
	\begin{equation*}
	\begin{split}
	\E&[(\bfz^{\top} \bfA \bfz -\E[\bfz^{\top} \bfA \bfz])^3]=  
	8 \beta_{2,2,2}\tr(\bfA^3) +(\beta_{2,2,2}+2\beta_{2}^3-3\beta_{2}\beta_{2,2}) (\tr \bfA)^3 \\
	&+6(\beta_{2,2,2}-\beta_{2}\beta_{2,2}) \tr \bfA \tr(\bfA^2)
	+3(\beta_{4,2}-\beta_{4}\beta_{2}+3 \beta_{2,2} \beta_{2}-3\beta_{2,2,2}) \tr \bfA \tr(\bfA \circ \bfA)\\
	&+12 (\beta_{4,2}-3\beta_{2,2,2}) \tr(\bfA \circ \bfA^2) 
	+ (\beta_6-15 \beta_{4,2} +30 \beta_{2,2,2}) \tr (\bfA \circ \bfA \circ \bfA)\,,
	\end{split}
  \end{equation*}
	where $\circ$ denotes the Hadamard product. Moreover,
	\begin{equation*}
	\begin{split}
	\E[(\bfz^{\top} \bfA \bfz )^3]&=  
	\beta_{2,2,2}[(\tr \bfA)^3 +6 \tr \bfA \tr(\bfA^2)+8 \tr(\bfA^3)]+ (\beta_6 -15 \beta_{4,2}+30\beta_{2,2,2}) \tr(\bfA \circ \bfA \circ \bfA )\\
	&\quad +(\beta_{4,2}-3\beta_{2,2,2}) [3\tr \bfA \tr(\bfA \circ \bfA)+12 \tr(\bfA \circ \bfA^2)]\,.
	\end{split}
  \end{equation*}
		If $\bfB$ is another real-valued and symmetric $n\times n$ nonrandom matrix, one has
\begin{equation*}
	\E[\bfz^{\top} \bfA \bfz \bfz^{\top} \bfB \bfz]=  \beta_{2,2} [\tr \bfA \tr \bfB +2 \tr(\bfA \bfB)] + (\beta_4-3\beta_{2,2}) \tr(\bfA \circ \bfB)\,.
  \end{equation*}
\end{lemma}

{\small
\bibliography{libraryDec2021}}

\end{document}